\documentclass[11pt]{amsart}
\usepackage{amsmath}
\usepackage{amscd}
\usepackage{url}
\usepackage{listings}

\usepackage{amssymb}
\usepackage{amsxtra}

\usepackage{mathabx}
\usepackage{color}
\usepackage{mathrsfs}
\usepackage{stmaryrd}
\usepackage[utf8]{inputenc}
\usepackage{eucal}
\usepackage{fullpage}
\usepackage{hyperref}

\def\Dbar{\leavevmode\lower.6ex\hbox to 0pt{\hskip-.23ex \accent"16\hss}D}
  \def\cftil#1{\ifmmode\setbox7\hbox{$\accent"5E#1$}\else
  \setbox7\hbox{\accent"5E#1}\penalty 10000\relax\fi\raise 1\ht7
  \hbox{\lower1.15ex\hbox to 1\wd7{\hss\accent"7E\hss}}\penalty 10000
  \hskip-1\wd7\penalty 10000\box7}
  \def\cfudot#1{\ifmmode\setbox7\hbox{$\accent"5E#1$}\else
  \setbox7\hbox{\accent"5E#1}\penalty 10000\relax\fi\raise 1\ht7
  \hbox{\raise.1ex\hbox to 1\wd7{\hss.\hss}}\penalty 10000 \hskip-1\wd7\penalty
  10000\box7} \def\cprime{$'$}

\renewcommand{\leq}{\leqslant}
\renewcommand{\geq}{\geqslant}
 
\numberwithin{equation}{section}

\newcommand{\N}{\mathbf{N}}
\newcommand{\E}{\mathbf{E}}

\newcommand{\Z}{\mathbf{Z}}
\renewcommand{\P}{\mathbf{P}}

\newcommand{\Q}{\mathbf{Q}}

\newcommand{\F}{\mathbf{F}}

\newcommand{\mods}[1]{\,(\mathrm{mod}\,{#1})}

\DeclareMathOperator{\Vol}{Vol}

\DeclareMathOperator{\res}{Res}

\DeclareMathOperator{\Imag}{Im}
\DeclareMathOperator{\real}{Re}

\DeclareMathOperator{\tr}{\mathrm{Tr}}

\DeclareMathOperator{\Aut}{Aut}

\DeclareMathOperator{\cond}{\mathbf{c}}
\DeclareMathOperator{\Ad}{Ad}

\newcommand{\eps}{\varepsilon}
\renewcommand{\rho}{\varrho}

\DeclareMathOperator{\SL}{SL}
\DeclareMathOperator{\GL}{GL}

\DeclareMathSymbol{\gena}{\mathord}{letters}{"3C}
\DeclareMathSymbol{\genb}{\mathord}{letters}{"3E}

\def\sums{\mathop{\sum \Bigl.^{*}}\limits}

\newcounter{bnd}

\theoremstyle{plain}
\newtheorem{theorem}{Theorem}[section]
\newtheorem*{theorem*}{Theorem}
\newtheorem{lemma}[theorem]{Lemma}
\newtheorem{corollary}[theorem]{Corollary}

\newtheorem{proposition}[theorem]{Proposition}

\theoremstyle{remark}

\theoremstyle{definition}

\renewcommand{\geq}{\geqslant}
\renewcommand{\leq}{\leqslant}

\begin{document}
 \title{Bounds for traces of Hecke operators and applications to modular and elliptic curves over a finite field}
 \author{Ian Petrow}
 \address{ETH Z\"urich - Departement Mathematik \\
HG G 66.4 \\
R\"amistrasse 101 \\
8092 Z\"urich \\
Switzerland }
\email{ian.petrow@math.ethz.ch}
\thanks{This work was supported by Swiss National Science Foundation grant PZ00P2\_168164.}

 \maketitle
 
 \begin{abstract}
 We give an upper bound for the trace of a Hecke operator acting on the space of holomorphic cusp forms with respect to certain congruence subgroups.  Such an estimate has applications to the analytic theory of elliptic curves over a finite field, going beyond the Riemann hypothesis over finite fields.  As the main tool to prove our bound on traces of Hecke operators, we develop a Petersson formula for newforms for general nebentype characters.
 \end{abstract}

\section{Introduction}
\subsection{Statement of Results}
 Let $S_\kappa(\Gamma,\epsilon)$ be the space of holomorphic cusp forms of weight $\kappa$, for a subgroup $\Gamma$ of a Hecke congruence group, and of nebentype character $\epsilon$.  We write $\tr( T | S_\kappa(\Gamma,\epsilon))$ for the trace of a linear operator $T$ acting on $S_\kappa(\Gamma,\epsilon)$.  The aim of this paper is to give estimates for $\tr(T_m | S_\kappa(\Gamma,\epsilon)),$ where $T_m$ is the $m$th Hecke operator, as the parameters $m,\kappa,\Gamma,$ and $\epsilon$ vary simultaneously.
 
 Consider first the case that $\Gamma = \Gamma_0(N)$ and $\epsilon$ is any Dirichlet character modulo $N$.  Let $d(m)$ denote the number of divisors of $m$, $\sigma(m)$ the sum of the divisors of $m$, and let $\psi(N) = [\Gamma_0(N):\SL_2(\Z)]= N \prod_{p\mid N}\big(1+ \frac{1}{p}\big)$.  We assume that $\kappa\geq 2$ an integer throughout the paper. Deligne's theorem tells us that each eigenvalue $\lambda(m)$ of $T_m$ satisfies $|\lambda(m)| \leq d(m)m^\frac{\kappa-1}{2}$.  Therefore we have the ``trivial'' estimate on the trace \begin{equation}\label{eq:trivial}\tr(T_m | S_\kappa(\Gamma_0(N),\epsilon)) \leq  \dim S_\kappa(\Gamma_0(N),\epsilon) d(m) m^{\frac{\kappa-1}{2}}\leq \frac{(\kappa-1) \psi(N)}{12} d(m) m^{\frac{\kappa-1}{2}}.\end{equation} For the bound on $\dim S_\kappa(\Gamma_0(N),\epsilon)$, see e.g.~\cite[Cor 8]{Ross}.  The power of $m$ in \eqref{eq:trivial} is sharp by the Sato-Tate distribution for Hecke eigenvalues.
 On the other hand, by a careful analysis using the Eichler-Selberg trace formula, Conrey, Duke and Farmer \cite{ConreyDukeFarmer} and in more generality Serre \cite[Prop.~4]{SerreVertical} show that if $\epsilon(-1)=(-1)^\kappa$ then
\begin{equation}\label{eq:Serre}\tr(T_m | S_\kappa(\Gamma_0(N),\epsilon))= \frac{\kappa-1}{12} \epsilon(m^{\frac{1}{2}}) m^{\frac{\kappa}{2} - 1} \psi(N) + O\left( (\sigma(m) \max_{f^2<4m} \psi(f) + d(m) N^{\frac{1}{2}}) m^{\frac{\kappa-1}{2}} d(N)\right),\end{equation}
where $\epsilon(m^{1/2})$ is understood to be $0$ unless $m$ is a perfect square. If $\epsilon(-1)\neq (-1)^\kappa$ then $S_\kappa(\Gamma_0(N),\epsilon)= \{0\}$, so the left hand side vanishes identically.  We expect the estimate \eqref{eq:Serre} to be sharp if $m$ is fixed and $\kappa + N \to \infty$.  

Write $\cond(\epsilon)$ for the conductor of the Dirichlet character $\epsilon$, and $\cond^*(\epsilon) = \prod_{p \mid \cond(\epsilon)} p$ for its square-free part. In this paper we prove:

\begin{theorem}\label{MT1}
Suppose that $\epsilon(-1)=(-1)^\kappa$, $(N,m)=1$, and that $m\cond(\epsilon) \cond^*(\epsilon)\ll (N^4\kappa^{10/3})^{1-\eta}$ for some $\eta>0$.  Then we have that \begin{equation}\label{eq:peterssonbound}\tr(T_m | S_\kappa(\Gamma_0(N),\epsilon)) = \frac{\kappa-1}{12} \epsilon(m^{\frac{1}{2}}) m^{\frac{\kappa}{2} - 1} \psi(N) + O_{\eta,\eps}\left(N^{\frac{10}{11}}m^{\frac{\kappa-1}{2}+\frac{1}{44}}\kappa^{\frac{61}{66}}\cond(\epsilon)^{\frac{1}{44}} \cond^*(\epsilon)^{\frac{1}{44}}(Nm\kappa )^\eps\right).\end{equation}
\end{theorem}
We remark that the hypothesis that $m\cond(\epsilon) \cond^*(\epsilon)\ll (N^4\kappa^{10/3})^{1-\eta}$ for some $\eta>0$ in Theorem \ref{MT1} is no restriction in practice, since if the hypothesis fails then \eqref{eq:trivial} is a superior bound anyway. Indeed, the error term in \eqref{eq:peterssonbound} is smaller than that in both \eqref{eq:trivial} and \eqref{eq:Serre} when $$N^{\frac{8}{13}}\kappa^{\frac{122}{195}}(N\kappa)^\eps \cond(\epsilon)^{\frac{1}{65}} \cond^*(\epsilon)^{\frac{1}{65}} \ll m \ll \frac{(N^4\kappa^{\frac{10}{3}})^{1-\eta}}{\cond(\epsilon) \cond^*(\epsilon)}.$$ For example, if $\epsilon$ is trivial and the weight $\kappa$ is fixed, then \eqref{eq:peterssonbound} is better than \eqref{eq:trivial} and \eqref{eq:Serre} for $$N^{\frac{8}{13}+\eps} \ll m \ll N^{4-\eps}.$$
Note that our result requires the hypothesis $(N,m)=1$, whereas the estimates \eqref{eq:trivial} and \eqref{eq:Serre} do not.  We discuss the source of this condition in the sketch of the proof, below.

We are also interested in spaces of modular forms for groups other than $\Gamma_0(N)$.  In particular, for positive integers $M \mid N$ let 
\begin{equation}\Gamma(M,N) = \big\{ \left(\begin{smallmatrix} a& b \\ c & d \end{smallmatrix}\right) \in \SL_2(\Z) \text{ s.t. } a,d\equiv 1 \mods N, \,\,\, c \equiv 0 \mods {NM}\big\}.\end{equation}   These congruence groups interpolate between $\Gamma_1(N)= \Gamma(1,N)$ and $\Gamma(N) \simeq \Gamma(N,N).$  We write $S_\kappa(M,N)$ for the space of modular forms of weight $\kappa$ for the group $\Gamma(M,N)$ (without nebentype character).  Let $\delta(a,b)$ be the indicator function of $a=b$ and $\delta_c(a,b)$ be the indicator function of $a\equiv b \mods c$.  Let $T_m$ be the $m$th Hecke operator acting on $S_\kappa(M,N)$ and for $(d,N)=1$ let $\langle d \rangle$ the $d$th diamond operator.  These operators commute and $T_1 =\langle 1 \rangle = \text{id}$; for definitions see \cite[\S5.1, 5.2]{DiamondShurman} or \cite[\S4]{KaplanPetrow2}.  In particular, we have
\begin{equation}\label{eq:mnstart}\tr (\langle d \rangle T_m | S_\kappa( \Gamma(M,N)) )  = \sum_{\epsilon \mods  N} \epsilon(d)\tr(T_m | S_\kappa(\Gamma_0(NM),\epsilon)).\end{equation}
Applying \eqref{eq:trivial} to \eqref{eq:mnstart} we have \begin{equation}\label{trivial2}\tr (\langle d \rangle T_m | S_\kappa( \Gamma(M,N)) )   \leq  \frac{\kappa -1}{12} \varphi(N) \psi(NM)d(m) m^{\frac{\kappa-1}{2}}.\end{equation}   Meanwhile, summing \eqref{eq:Serre} over characters $\epsilon \mods N$ such that $\epsilon(-1)=(-1)^\kappa$ we find \begin{multline}\label{serre2}  \tr(  \langle d \rangle T_m | S_\kappa(\Gamma(M,N))) = \frac{\kappa-1}{24}m^{\frac{\kappa}{2}-1} \varphi(N) \psi(NM) \left(\delta_N(m^{\frac{1}{2}} d, 1) + (-1)^\kappa \delta_N(m^{\frac{1}{2}} d, -1)\right) \\ + O\Big( (\sigma(m) \max_{f^2<4m} \psi(f) + d(m) (MN)^{\frac{1}{2}}) m^{\frac{\kappa-1}{2}} d(MN)N \Big)  . \end{multline}
The following result improves on both \eqref{trivial2} and \eqref{serre2} in an intermediate range of parameters.

\begin{theorem}\label{MT2}
Suppose that $M \mid N$, $(N,m)=1$, and that $m \ll (N^6\kappa^{10/3})^{1-\eta}$ for some $\eta>0$. We have that \begin{multline*} \tr( \langle d \rangle T_m | S_\kappa(\Gamma(M,N)))= \frac{\kappa-1}{24}m^{\frac{\kappa}{2}-1} \varphi(N) \psi(NM) \left(\delta_N(m^{\frac{1}{2}} d, 1) + (-1)^\kappa \delta_N(m^{\frac{1}{2}} d, -1)\right) \\ + O_{\eta,\eps}\Big(M N^{\frac{41}{22}}m^{\frac{\kappa-1}{2}+\frac{1}{44}}\kappa^{\frac{61}{66}}(Nm\kappa)^{\eps}\Big).\end{multline*}
\end{theorem}

\subsection{Applications to Modular and Elliptic Curves over a Finite Field}
Hecke operators appear throughout number theory, and estimates for their traces are especially relevant to equidistribution problems.  See for example \cite[\S5-\S8]{SerreVertical} and \cite{MurtySinha}. We mention here a few consequences in the analytic theory of modular and elliptic curves over a finite field.  

Let $C$ be a nonsingular projective curve of genus $g$ over a finite field $\F_q$ with $q$ elements.  Then we have (see e.g.~\cite[Ch.~11]{milneMF}) that $$| C(\F_{q^n})| = q^n + 1 - \sum_{i=1}^{2g} \alpha_i^n,$$ where $\{\alpha_i\}$ are the inverse zeros of the zeta function of $C$ $$ Z(C,T)= \frac{(1-\alpha_1 T) \cdots (1-\alpha_{2g}T)}{(1-T)(1-qT)}.$$ The Riemann hypothesis for curves over finite fields asserts that $|\alpha_i|= \sqrt{q}$ for all $i$.  Igusa \cite{IgusaKroneckerian} showed that there exists a non-singular projective model for $X_0(N)$ over $\Q$ whose reductions modulo primes $p$, $p\nmid N$ are also non-singular (see also the survey \cite[\S 9]{DiamondIm}), and so the preceding discussion applies to $X_0(N)$ when $p\nmid N$.  Since $g \sim \psi(N)/12$ as $N\to \infty$ we have that \begin{equation}\label{RHonly}|X_0(N)(\F_q)| = q+1 + O (\psi(N) q^{1/2}).\end{equation} In particular, $|X_0(N)(\F_q)| \sim q$ as $q\to \infty$ as soon as $q\gg N^{2+\delta}$ for some $\delta>0$.  On the other hand, the Eichler-Shimura correspondence (see e.g.~\cite[Thm.~11.14]{milneMF}) asserts that $$Z(X_0(N),T) =\frac{ \prod_{f \in H_2(N)} \left(1-\lambda_f(p) T + pT^2\right)}{(1-T)(1-pT)},$$ where $H_2(N)$ is a basis for $S_2(\Gamma_0(N))$ consisting of eigenforms of $\{T_p, p \nmid N\}$, and $\lambda_f(p)$ is the $T_p$ eigenvalue of $f$.  We therefore have $$|X_0(N)(\F_q)| = q+1 - \tr(T_q | S_2(\Gamma_0(N))) + p \tr(T_{q/p^2} | S_2(\Gamma_0(N))),$$ where we set $T_{p^{-1}}=0$.  Applying \eqref{eq:trivial}, \eqref{eq:Serre}, and Theorem \ref{MT1} we get \begin{corollary}\label{MC1}Suppose $q=p^v$ is a prime power such that $p \nmid N$.  We have $$|X_0(N)(\F_q)| = q+(p-1)\frac{\psi(N)}{12}\delta_2(v,0) + O_{\eps}\left( \min(\psi(N), q^{\frac{1}{44}}N^{\frac{10}{11}}(qN)^\eps, (q^{\frac{3}{2}}+N^{\frac{1}{2}})d(N)q^\eps )q^{\frac{1}{2}}\right).$$ In particular, the main term is larger than the error term as soon as $q \gg N^{\frac{40}{21}+\delta}$ for some fixed $\delta>0$. \end{corollary}  Corollary \ref{MC1} shows that there is significant cancellation between the zeros $\alpha_i$ of $Z(X_0(N),T)$, and in this sense goes beyond the Riemann hypothesis for $Z(X_0(N),T)$.  Assuming square-root cancellation between the zeros, one might conjecture an error term of size $(qN)^{1/2+\eps}$ in Corollary \ref{MC1}, which would imply that the main term is larger than the error term whenever $q \gg N^{1+\delta}$ for some $\delta$. If one assumes the generalized Lindel\"of hypothesis for adjoint square $L$-functions, then the method in this paper produces an error term of size $q^{1/8+\eps}N^{1/2+\eps}$ in Corollary \ref{MC1} (see Lemma \ref{AFE}). In a much more speculative direction, if under the assumption $(mn,W)=1$ the upper bound $\ll_{\kappa,\eps} (mnW)^\eps W^{-1/2}$ for the sum appearing in Lemma \ref{ILSlem} holds (cf. the Linnik-Selberg conjecture), then the error term $(qN)^{1/2+\eps}$ in Corollary \ref{MC1} is admissible.

If $q$ is a square then we can compare the second main term in Corollary \ref{MC1} to the error term coming from \eqref{eq:Serre} in the range where $q$ is small compared to $N$. For example, in the special case that $p$ is a prime and $q=p^2$ we have 
\begin{corollary}
If $p,N \to \infty$ where $p$ runs through primes $p \nmid N$ then for any fixed $\delta>0$ we have 
$$|X_0(N)(\F_{p^2})| = \begin{cases} p^2 + O(p\psi(N)) & \text{ if } p^2 \gg N^{4-\delta} \\
p^2 + p \frac{\psi(N)}{12} + O_\eps (p^{\frac{23}{22}}N^{\frac{10}{11}}(qN)^\eps ) & \text{ if } N^{\frac{40}{21}-\delta} \ll p^2 \ll N^{4-\delta} \\
p \frac{\psi(N)}{12} + O_\eps (p^{\frac{23}{22}}N^{\frac{10}{11}}(qN)^\eps ) & \text{ if } N^{\frac{8}{13}+\delta} \ll p^2 \ll N^{\frac{40}{21}-\delta} \\
(p-1) \frac{\psi(N)}{12}  + O_\eps((p^4+N^{\frac{1}{2}}p)d(N)p^\eps) & \text{ if } p^2 \ll N^{\frac{2}{3}-\delta}.\end{cases}$$
\end{corollary}
The first of these cases is just \eqref{RHonly}, and the last is the Tsfasman-Vl\u adu\c t-Zink theorem \cite{TsfasmanVladutZink}, which has important applications to algebraic coding theory, see \cite[Ch.~5]{Moreno}. 

Using Theorem \ref{MT2} we can make more explicit statements about elliptic curves themselves.
Let $E$ be an elliptic curve defined over $\F_q$ and let $t_E= q+1-\#E(\F_q)$ be the trace of the associated Frobenius endomorphism.  Hasse's Theorem tells us that $|t_E| \leq 2 \sqrt{q}$. The set of $\F_q$-isomorphism classes of elliptic curves defined over $\F_q$ is naturally a probability space where the probability of a singleton is given by $$\P_q(\{E\}) = \frac{1}{q |\Aut_{\F_q}(E)|}.$$ We would like to study the expectations as $q\to \infty$ of various random variables associated to $t_E$ or the structure of the group of $\F_q$-rational points of $E$.  To be precise: let $A$ be a finite abelian group with at most two generators, and let $\Phi_A$ denote the indicator function of the event that there exists an injective group homomorphism $A\hookrightarrow E(\F_q)$.  Let $U_j(x)$ for $j\geq 0$ be the Chebyshev polynomials of the second kind.  The Chebyshev polynomials form an orthonormal basis for the Hilbert space $L^2([-1,1],\frac{2}{\pi} \sqrt{1-x^2}dx)$. N.~Kaplan and the author in  \cite[Thm.~2]{KaplanPetrow2} gave explicit formulas for the expectations $$ \E_{q}(U_j(t_E/2\sqrt{q}) \Phi_A) = \frac{1}{q} \sum_{\substack{E/\F_q \\ A \hookrightarrow E(\F_q)}} \frac{U_j(t_E/2 \sqrt{q})}{|\Aut_{\F_q}(E)|}$$ in terms of $\tr( \langle d \rangle T_m | S_\kappa(\Gamma(M,N)))$ and elementary arithmetic functions of $m,M,N,$ and $j$.  

Theorem \ref{MT2} yields the following refinement of the error term in the main corollary of \cite{KaplanPetrow2}. Let $$v(n_1,n_2) = \frac{n_1}{\psi(n_1) \varphi(n_1)n_2^2} \prod_{\ell \mid \frac{n_1}{(q-1,n_1)}} \left(1+\ell^{-1-2v_\ell\left(\frac{(q-1,n_1)}{n_2}\right)} \right).$$

\begin{corollary}\label{MC2}
Let $n_1=n_1(A)$ and $n_2=n_2(A)$ be the first and second invariant factors of $A$ (i.e.~we have $n_2 \mid n_1$).  Suppose that $(|A|,q)=1$ and $q \equiv 1 \mods {n_2}$.  Then $$\E_{q}(U_{j}(t_E/2\sqrt{q}) \Phi_A)  = v(n_1,n_2)\left( \delta(j,0) + O_{j,\eps}\left(\min( n_1 ,q^{\frac{1}{44}} n_1^{\frac{19}{22}})  n_1n_2q^{-\frac{1}{2}}(qn_1)^\eps \right)\right).$$ If $q \not \equiv 1 \mods {n_2}$, then $\E_q(U_j\Phi_A)$ vanishes identically. 

In particular, the traces of the Frobenius $t_E$ for $\{E/\F_q : A \hookrightarrow E(\F_q)\}$ become equidistributed with respect to the Sato–Tate measure as $q \to \infty$ through prime powers $q \equiv 1 \mods{n_2}$. The equidistribution is uniform in A as soon as $q\gg n_2^2 n_1^{\frac{41}{11}+\delta}$ for any fixed $\delta>0$. 
\end{corollary}
In \cite{KaplanPetrow2} Kaplan and the author showed that the equidistribution of $t_E$ for $\{E/\F_q : A \hookrightarrow E(\F_q)\}$ is uniform as soon as $q\gg n_2^2 n_1^{4+\delta}$ by applying \eqref{trivial2} to bound the trace.  In this sense, Corollary \ref{MC2} goes beyond what one can conclude using the Riemann hypothesis of Deligne alone.  All of the error terms in the theorems and corollaries found in section 2 of \cite{KaplanPetrow2} are similarly improved by applying Theorem \ref{MT2} in addition to \eqref{trivial2}.
\subsection{Outline of Proof}\label{proofoutline}
Thanks to \eqref{eq:mnstart}, the structural steps of the proof of Theorem \ref{MT2} reduce to those of Theorem \ref{MT1}.  The details of the analytic arguments differ however (see section \ref{AforMN}).  For these reasons, we only discuss the proof of Theorem \ref{MT1} in this outline.

By Atkin-Lehner theory, to estimate $\tr(T_m | S_\kappa(\Gamma_0(N),\epsilon))$ it suffices to estimate \begin{equation}\label{sum}\sum_{f \in H^\star_\kappa(N,\epsilon)}\lambda_f(m),\end{equation} where $H^\star_\kappa(N,\epsilon)$ is set of Hecke-normalized newforms of level $N$ and character $\epsilon$, and $\lambda_f(m)$ is the $m$th Hecke eigenvalue of $f$, normalized so that $|\lambda_f(n)| \leq d(n)$.  Whereas Serre and Conrey, Duke, and Farmer used the Eichler-Selberg trace formula to access the trace of $T_m$, we take a different path and use the Petersson trace formula.  

Let $\mathcal{B}_\kappa(\Gamma_0(N),\epsilon)$ be an orthonormal basis for $S_\kappa(\Gamma_0(N),\epsilon)$.  Let $g \in \mathcal{B}_\kappa(\Gamma_0(N),\epsilon)$ and write its Fourier coefficients as $\{b_g(n)\}_{n\geq 1}$. Then the Petersson formula \cite[Prop.~14.5]{IK} says that \begin{equation}\label{eq:petersson}\frac{\Gamma(\kappa-1)}{( 4 \pi \sqrt{mn})^{\kappa-1}} \sum_{f \in \mathcal{B}_\kappa(\Gamma_0(N),\epsilon)} b_f(n)\overline{b_f(m)} = \delta(m,n) + 2\pi i^{-\kappa} \sum_{\substack{c>0 \\ c \equiv 0 \mods N}} \frac{S_{\epsilon}(m,n,c)}{c} J_{\kappa-1}\left( \frac{4\pi \sqrt{mn}}{c}\right),\end{equation} where $J_{\alpha}$ is the $J$-Bessel function, $S_\epsilon(m,n,c)$ is the twisted Kloosterman sum $$S_\epsilon(m,n,c) = \sums_{d \mods c}\epsilon(d) e\left( \frac{dm + \overline{d}n}{c}\right),$$ and the $*$ indicates we run over invertible $d \mods c$.

Our goal is to apply the Petersson formula to \eqref{sum}, and so we are faced with two technical difficulties: \begin{enumerate} \item Only the newforms in $S_\kappa(\Gamma_0(N),\epsilon)$ have Fourier coefficients proportional to the Hecke eigenvalues appearing in \eqref{sum}, and \item If $f$ is a newform, the constant of proportionality between Fourier coefficients $b_f(n)$ and the Hecke eigenvalues $\lambda_f(n)$ is $\approx ||f||_{L^2}$, which is not constant across $H^\star_{\kappa}(N,\epsilon).$\end{enumerate}

We overcome (1) in Theorem \ref{thm:petersson} by developing a Petersson formula for newforms for $S_\kappa(\Gamma_0(N),\epsilon)$.  There has been much recent interest in such formulas, see for example \cite{BBDDM}, \cite{NelsonPeterssonFormula}, \cite{PetrowYoung1}, and \cite{YoungEis}.  Theorem \ref{thm:petersson} is a generalization of \cite[Prop.~4.1]{BBDDM} to nontrivial central characters, which itself is a generalization of work of Iwaniec-Luo-Sarnak \cite{ILS}, Rouymi \cite{Rouymi} and Ng \cite{NgBasis}. Peter Humphries has also shared a preprint with the author in which he independently obtains Theorem \ref{thm:petersson}, and uses it to study low-lying zeros of the $L$-functions associated to $f \in H^\star_\kappa(N,\epsilon)$.  Theorem \ref{thm:petersson} is the only place in the proof where we have used the hypothesis $(N,m)=1$, in an essential way, and so is the source of the relatively prime conditions in Theorems \ref{MT1} and \ref{MT2}.  

We deal with (2) by appealing to the special value formula $$L(1,\Ad^2 f) = \frac{\zeta^{(N)}(2) (4 \pi)^{\kappa}}{\Gamma(\kappa)} \frac{|| f||^2_{L^2}}{\Vol X_0(N)},$$ where $L(s,\Ad^2 f)$ is a certain Dirichlet series whose coefficients involve $\lambda_f(n^2)$, and which we discuss in more detail in section \ref{background}.  One may then swap the sum over $f$ and this Dirichlet series, and apply our Petersson formula for newforms (Theorem \ref{thm:petersson}).  Estimating the resulting sums directly using the Weil bound for $S_\epsilon(a,b,c)$ (see Lemma \ref{Weil}), one recovers that the trace of $T_m$ is $\ll_m (N\kappa)^{1+\eps}$ (compare with \eqref{eq:trivial}).  

To save a bit more and obtain Theorem \ref{MT1} we remove the weights $||f||^2_{L^2}$ more efficiently using a method due to Kowalski and Michel \cite[Prop.~2]{KMrankofJ0}.  Kowalski and Michel's method is based on H\"older's inequality and a large sieve inequality due to Duke and Kowalski \cite[Thm.~4]{DukeKowalski} for sub-families of automorphic forms on $\GL_3$.  There are other notable large sieve inequalities for $\GL_3$ in the literature, see e.g.~\cite[Thm.~3]{BBMlargesieve} and \cite[Thm.~1]{Vlargesieve}.  However, these two are not useful to us since we need a large sieve inequality which is efficient for the proper sub-family of $\GL_3$ forms cut out by the image of the adjoint square lift from $\GL_2$.  The inequality of Duke and Kowalski is superior to the results \cite[Thm.~3]{BBMlargesieve} and \cite[Thm.~1]{Vlargesieve} in the case of a thin subfamily and a long summation variable, which is the situation of interest to us.
\subsection{Acknowledgements}
I would like to thank Nathan Kaplan for a careful read and pointing out the Tsfasman-Vl\u adu\c t-Zink theorem to me, Corentin Perret-Gentil for some helpful discussions, and the anonymous referee for a thorough and detailed report on the first version of this paper.

\section{Preliminaries on $L$-series}\label{background}

If $L(s)$ is a meromorphic function defined in $\real(s) \gg 1$ by an infinite product over primes $p$ of local factors $L_p(s)$, then for any integer $N$ we write $$L^{(N)}(s) = \prod_{p \nmid N}L_p(s)$$ and $$L_N(s) = \prod_{p\mid N} L_p(s),$$ so that $L(s) = L_N(s)L^{(N)}(s)$ for any $N\in \N$.  To deal with the $||f||_{L^2}^2$-normalization alluded to in subsection \ref{proofoutline}, we introduce the ``naive'' adjoint square $L$-function. 
For $f \in H_\kappa^\star(N,\epsilon)$, let $$L(s,\Ad^2 f) = \frac{\zeta^{(N)}(2s)}{\zeta(s)} \sum_{n\geq 1} \frac{|\lambda_f(n)|^2}{n^s}= \prod_p L_p(s,\Ad^2 f),$$ where $\zeta(s)$ is the Riemann zeta function, and where \begin{equation}\label{Adformula} L_p(s,\Ad^2 f) = \begin{cases}  
\left(1-\frac{1}{p^{2s}}\right)^{-1} \sum_{\alpha \geq 0 } \frac{\overline{\epsilon}(p^\alpha) \lambda_f(p^{2\alpha})}{p^{\alpha s}}  & \text{ if } p \nmid N \\
\left(1- \frac{1}{p^s}\right) \left( 1- \frac{|\lambda_f(p)|^2}{p^s}\right)^{-1}  & \text{ if } p \mid N. \end{cases} \end{equation}
Warning: the $L(s,\Ad^2 f)$ is \emph{not} the true adjoint square $L$-function of $f$ as defined by functoriality (see \cite[pp.~133]{IK} and the online errata).  But if $p \nmid N$, then $L_p(s,\Ad^2 f)$ \emph{does} match the local $L$-factor at $p$ of the true adjoint square $L$-function.  Our ``naive'' adjoint square $L$ function $L(s,\Ad^2 f)$ is chosen to be the Dirichlet series for which the following Lemma is true.  
\begin{lemma}\label{Rankin} 
The series $L(s,\Ad^2 f)$ defined above is holomorphic for $\real(s) >0 $ and \begin{equation}\label{eq:L1Ad} L(1,\Ad^2 f) = \frac{\zeta^{(N)}(2) (4 \pi)^{\kappa}}{\Gamma(\kappa)} \frac{\langle f,f\rangle_{N}}{\Vol X_0(N)},\end{equation} where $$\langle f,f\rangle_N = \int_{\Gamma_0(N) \backslash \mathcal{H}} |f(z)|^2 y^{\kappa}\,\frac{dx\,dy}{y^2}$$ and $$\Vol X_0(N) = \int_{\Gamma_0(N)  \backslash \mathcal{H} } \frac{dx\,dy}{y^2} = \frac{\pi}{3} \psi(N).$$  \end{lemma}
\begin{proof} For the first statement, let $\pi$ denote the irreducible admissible cuspidal automorphic representation of $\GL_2$ generated by $f$, and denote by $L(s, \Ad^2 \pi)$ the $L$-function of its adjoint square lift.  We have by Gelbart and Jacquet \cite{GelbartJacquet} that $L(s,\Ad^2 \pi)$ is an entire function of $s$.  Therefore, the prime-to-$N$ part of the naive $L$-function $L^{(N)}(s,\Ad^2 f)$ is holomorphic for $\real(s)>0$.  

For the second statement, take the standard non-holomorphic Eisenstein series for $\Gamma_0(N)$ at the cusp $\infty$ given by $$E(z,s) = \sum_{\gamma \in \Gamma_\infty \backslash \Gamma_0(N)} \Imag(\gamma z)^s.$$ Then we have by the classical Rankin-Selberg unfolding argument $$ \int_{\Gamma_0(N) \backslash \mathcal{H}} |f(z)|^2 E(z,s) y^\kappa \, \frac{dx\,dy}{y^2} = \frac{\Gamma(s+ \kappa -1)}{(4 \pi )^{s+\kappa -1}} \sum_{n\geq 1} \frac{|\lambda_f(n)|^2}{n^s}.$$ We deduce the lemma by taking residues on both sides and recalling \cite[Thm.~13.2]{ClassIw} that $$\res_{s=1} E(z,s) = \Vol X_0(N)^{-1}.$$ 
\end{proof}  
Let $\rho_f(n)$ be the Dirichlet series coefficients of $L^{(N)}(s,\Ad^2 f)$.  Explicitly, \begin{equation}\label{eq:rho}\rho_f(n) = \begin{cases} \sum_{n=m^2 \ell} \overline{\epsilon}(\ell) \lambda_f(\ell^2) & \text{ if } (n,N)=1 \\
0 & \text{ if } (n,N)>1. \end{cases}\end{equation} Inverting, we also have \begin{equation}\label{eq:rhoinverse}\overline{\epsilon}(n) \lambda_f(n^2) = \sum_{m^2 \ell=n} \mu(m) \rho_f(\ell).\end{equation} For future reference, we write the partial sums of $L^{(N)}(1,\Ad^2 f)$ compactly as \begin{equation}\label{omegafx}\omega_f(x) = \sum_{n\leq x} \frac{\rho_f(n)}{n}.\end{equation}

By contrast, when $p \mid N$ we have that $L_p(s,\Ad^2 f)$ is constant along $f \in H_\kappa^\star(N,\epsilon)$ by the following Lemma. \begin{lemma}[\cite{Ogg} Thms 2, 3]\label{Ogg}
Let $p \mid N$ be a prime, and $\epsilon$ a Dirichlet character mod $N$.  Write $$a_{N,\epsilon}(p) = \begin{cases} 1 & \text{ if } \epsilon \text{ is \emph{not} a character mod } N/p \\ \frac{1}{p} & \text{ if } \epsilon \text{ is a character mod } N/p \text{ and } p^2 \nmid N \\ 0 & \text{ if } \epsilon \text{ is a character mod } N/p \text{ and } p^2 \mid N.\end{cases}$$ Then we have $|\lambda_f(p)|^2=a_{N,\epsilon}(p).$
\end{lemma}

\section{Structural Steps}
We study the operator $T'_m = T_m/m^{\frac{\kappa-1}{2}}$ on $S_\kappa(\Gamma_0(N),\epsilon)$ so that each eigenvalue $\lambda_f(m)$ of the $T'_m$ operator is normalized by Deligne's theorem to have $|\lambda_f(m)|\leq d(m)$.  We write $H^\star_\kappa(N,\epsilon)$ for the set of Hecke-normalized newforms in $S_\kappa(N,\epsilon)$ in the sense of Atkin-Lehner theory \cite{AtkinLehner,LiPrimitive}. 
Also by Atkin-Lehner theory we have when $(m,N)=1$ that \begin{equation}\label{eq:nestart}\tr(T_m' | S_\kappa(\Gamma_0(N),\epsilon))
=  \sum_{LM = N } d(L)\sum_{f \in H_\kappa^\star(M, \epsilon)} \lambda_f(m),\end{equation} where we consider the interior sum to be empty if $\epsilon$ is not a character mod $M$.  
Thanks to \eqref{eq:mnstart}, we can reduce the structural steps for traces on $S_\kappa( \Gamma(M,N))$ to the case of $S_\kappa(\Gamma_0(N),\epsilon)$.

 Recall the notation from section \ref{proofoutline} and write $c_\kappa = \Gamma(\kappa-1)/(4 \pi)^{\kappa-1}$. Let $$\Delta_{\kappa,N,\epsilon}(m,n) =  \frac{c_\kappa}{( \sqrt{mn})^{\kappa-1}} \sum_{f \in \mathcal{B}_\kappa(\Gamma_0(N),\epsilon)} b_f(n)\overline{b_f(m)} ,$$ so that the Petersson formula \eqref{eq:petersson} is \begin{equation}\label{petersson2}\Delta_{\kappa,N,\epsilon}(m,n) = \delta(m,n) + 2\pi i^{-\kappa} \sum_{\substack{c>0 \\ c \equiv 0 \mods N}} \frac{S_{\epsilon}(m,n,c)}{c} J_{\kappa-1}\left( \frac{4\pi \sqrt{mn}}{c}\right).\end{equation}  The following theorem is our main tool for computing sums over the set of newforms $H_\kappa^\star(N,\epsilon)$. \begin{theorem}\label{thm:petersson}
If $(mn,N)=1$ then we have $$c_\kappa \sum_{f \in H_\kappa^\star(N,\epsilon)} \frac{\overline{\lambda_f(m)}\lambda_f(n)}{\langle f,f \rangle_N} = \sum_{LM=N} \mu(L) R(M,L,\epsilon) \sum_{\substack{\ell \mid L^\infty \\ (\ell,M)=1}} \frac{\overline{\epsilon}(\ell)}{\ell} \Delta_{\kappa, M,\epsilon}(m,n\ell^2),$$ where $$R(M,L,\epsilon):= \frac{1}{L} \prod_{\substack{p^2 \mid L \\ p \nmid M}} \left(1- \frac{1}{p^2}\right)^{-1} \prod_{p \mid (M,L)} \left(1-\frac{a_{M,\epsilon}(p)}{p}\right)^{-1},$$ and $a_{M,\epsilon}(p)$ was defined in Lemma \ref{Ogg}. 
\end{theorem}
\begin{proof} See section \ref{petproof}.\end{proof}

Theorem \ref{thm:petersson} does not directly apply to \eqref{eq:nestart} because of the normalization by $\langle f, f\rangle_N$. 

We present a technique for removing the weights $\langle f, f\rangle_N$, which is a slight generalization of Kowalski and Michel \cite[\S 3]{KMrankofJ0}. The idea for removing such weights first appeared in a paper of Murty \cite{MurtyrankofJ0}. 
Let $\alpha = (\alpha_f)$ be a sequence of complex numbers indexed by $$f \in \bigcup_{N\geq 1} \bigcup_{\epsilon \mods N} H^\star_\kappa(N,\epsilon).$$  Define the natural averaging operator $$A[\alpha]=A_{N,\epsilon}[\alpha] = \sum_{f \in H^\star_\kappa(N,\epsilon)} \alpha_f.$$ Let $$\omega_f = c_\kappa \frac{L_N(1,\Ad^2 f)}{\langle f,f\rangle_N}.$$ Then we define the \emph{harmonic} averaging operator $$A^h[\alpha]=A^h_{N,\epsilon}[\alpha] =   \sum_{f \in H^\star_\kappa(N,\epsilon)} \omega_f \alpha_f.$$ 

The following proposition is a minor generalization of Proposition 2 of \cite{KMrankofJ0}. It allows us to pass from natural averages of newforms to harmonic averages of newforms. 

\begin{proposition}\label{prop:Holder}
Let $\alpha = (\alpha_f)$ be a sequence of complex numbers indexed by $f \in H_\kappa^\star(N,\epsilon)$ running over all $N$ and all $\epsilon$.  Suppose that for all $\eps>0$ \begin{equation}\label{eq:HkNeHypothesis1} A^h[|\alpha_f| ] \ll_\eps (N\kappa)^\eps \end{equation} and \begin{equation}\label{eq:HkNeHypothesis2}  \max_{f \in H_\kappa^\star(N,\epsilon)} \left| \omega_f \alpha_f \right| \ll (N\kappa)^{-\delta + \eps}\end{equation} for some absolute $\delta>0$.  For any integer $r\geq 1$ write $x=(N\kappa)^{\frac{10}{r}}$. Then we have $$A[\alpha_f] = \frac{\kappa-1}{4\pi} \frac{\Vol X_0(N)}{\zeta^{(N)}(2)} \left( A^h[ \omega_f(x) \alpha_f] + O_{\eps,r}(x^{-\frac{\delta}{20} +\eps} + (N\kappa)^{-1})\right).$$
\end{proposition}
\begin{proof} See section \ref{holderproof}.\end{proof}

One of the main ingredients in the proof of Proposition \ref{prop:Holder} is a large sieve inequality for the Dirichlet series coefficients of the automorphic adjoint square $L$-funciton $L(s,\Ad^2 \pi)$, see Proposition \ref{prop:largesieve}, which is a quotation of \cite[Cor 6]{DukeKowalski}. This inequality is only valid when the length of summation $X$ satisfies $X \gg (N \kappa)^8$, which is far from the expected truth. Nonetheless, as of now it is the best available such inequality in the range of parameters of interest to us. The exponent $-\delta/20$ in Proposition \ref{prop:Holder} is optimized given the exponent $8$ above, and any improvement over the result of Duke and Kowalski would lead to a corresponding improvement to the value $20=2(8+2)$.

We apply Proposition \ref{prop:Holder} with $\alpha_f = \overline{\lambda_f}(m)$ to equation \eqref{eq:nestart} to get \begin{multline}\label{eq:ne3}\overline{\tr(T'_m | S_\kappa(\Gamma_0(N),\epsilon))}=  \sum_{LM = N } d(L)\frac{\kappa-1}{4\pi} \frac{\Vol X_0(M)}{\zeta^{(M)}(2)} A_{M,\epsilon}^h[\omega_f(x)\overline{\lambda_f(m)}] + O \left(\kappa N^{1+\eps}x^{-\delta /20+\eps} + N^\eps\right) \\ =  \frac{\kappa-1}{12} \sum_{LM = N } \frac{d(L)\psi(M)}{\zeta^{(M)}(2)} \sum_{\substack{n \leq x \\ (n,M)=1}} \frac{1}{n} \sum_{n=k^2 \ell} \overline{\epsilon}(\ell) A_{M,\epsilon}^h[\overline{\lambda_f(m)} \lambda_f(\ell^2)] \\ + O \left(\kappa N^{1+\eps}x^{-\delta /20+\eps} + N^\eps\right) .\end{multline}
  
We are now ready to apply Theorem \ref{thm:petersson}.  We deduce a version of the newform formula for the harmonic averages $A^h[\overline{\lambda_f(m)}\lambda_f(n)]$ appearing in \eqref{eq:ne3}.
\begin{lemma} Let $\cond_p(\epsilon)$ denote the exponent of the $p$-part of $\cond(\epsilon)$.  If $(mn,N)=1$ then we have \begin{equation}\label{eq:newforms}\begin{split} A_{N,\epsilon}^h[\overline{\lambda_f(m)}\lambda_f(n)] =  \frac{1}{\psi(N)}\sum_{LM=N} \mu(L) M F(M,\epsilon) \prod_{p^2 \mid M} \left(1-\frac{1}{p^2}\right) \sum_{\substack{\ell \mid L^\infty \\ (\ell,M)=1}} \frac{\overline{\epsilon}(\ell)}{\ell} \Delta_{\kappa, M,\epsilon}(m,n\ell^2),\end{split}\end{equation} where $$F(M,\epsilon) = \prod_{\substack{p || M \\ \cond_p (\epsilon )= 1}} \left(1+\frac{1}{p} \right) \prod_{\substack{p^\alpha || M \\ \alpha \geq 2 \\  \cond_p (\epsilon) = \alpha}} \left(1-\frac{1}{p} \right)^{-1}.$$ In particular, if $\epsilon=\epsilon_0$ is trivial we have \begin{equation}\label{BBDDMformula}A_{N,\epsilon_0}^h[\overline{\lambda_f(m)}\lambda_f(n)]  = \frac{1}{\psi(N)}\sum_{LM=N} \mu(L)M \prod_{p^2 \mid M} \left(1-\frac{1}{p^2}\right) \sum_{\substack{\ell \mid L^\infty \\ (\ell,M)=1}} \frac{1}{\ell} \Delta_{\kappa, M,\epsilon_0}(m,n\ell^2).\end{equation}  \end{lemma} Note that formula \eqref{BBDDMformula} resembles closely the formula found in \cite[Prop.~4.1]{BBDDM}.
\begin{proof}
By the definition of $L_p(1,\Ad^2 f)$ and Theorem \ref{thm:petersson} we have $$A^h[\overline{\lambda_f(m)}\lambda_f(n)]  = \prod_{p\mid N} \left(1-\frac{1}{p}\right) \left(1-\frac{a_{N,\epsilon}(p)}{p}\right)^{-1} \sum_{LM=N} \mu(L) R(M,L,\epsilon) \sum_{\substack{\ell \mid L^\infty \\ (\ell,M)=1}} \frac{\overline{\epsilon}(\ell)}{\ell} \Delta_{\kappa, M,\epsilon}(m,n\ell^2).$$ 
It suffices to show for any $L,M$ that 
\begin{equation}\label{identitylem33} \frac{\psi(LM)}{M} \prod_{p \mid LM} \left(1 - \frac{1}{p}\right) \left(1 -\frac{a_{LM,\epsilon}(p)}{p}\right)^{-1} R(M,L,\epsilon) =  \prod_{p^2 \mid M}\left( 1 - \frac{1}{p^2}\right) F(M,\epsilon).\end{equation}
We may also assume that $c(\epsilon) \mid M$, since otherwise $\Delta_{\kappa,M,\epsilon}(m,n \ell^2)=0$. Both sides of \eqref{identitylem33} are multiplicative, so it suffices to check the case $M=p^\alpha$ and $L = p^\beta$ for an arbitrary prime $p$. The following cases can be easily verified one-by-one. \begin{itemize}
\item $\alpha \geq 2$, $\beta \geq 1$, and $\cond_p(\epsilon) = \alpha$
\item $\alpha \geq 2$, $\beta \geq 1$, and $\cond_p(\epsilon) < \alpha$
\item $\alpha \geq 2$, $\beta =0$, and $\cond_p(\epsilon) = \alpha$
\item $\alpha \geq 2$, $\beta =0$, and $\cond_p(\epsilon) < \alpha$
\item $\alpha =1$, $\beta \geq 1$, and $\cond_p(\epsilon) = 1$
\item $\alpha =1$, $\beta \geq 1$, and $\cond_p(\epsilon) = 0$
\item $\alpha =1$, $\beta =0$, and $\cond_p(\epsilon) = 1$
\item $\alpha =1$, $\beta =0$, and $\cond_p(\epsilon) = 0$
\item $\alpha =0$, $\beta \geq 2$, and $\cond_p(\epsilon) = 0$
\item $\alpha =0$, $\beta =1$, and $\cond_p(\epsilon) = 0$.
\end{itemize}
\end{proof}

\section{Analysis for $\Gamma_0(N)$}\label{AforG0}
Now we put together \eqref{eq:ne3}, the newform formula \eqref{eq:newforms}, and the Petersson formula \eqref{petersson2}. 
By \eqref{eq:ne3} and \eqref{eq:newforms} we have that $$\overline{\tr( T'_m | S_\kappa(\Gamma_0(N),\epsilon))} = A + E,$$ where for an integer $r\geq 1$ to be chosen later we set $x^r = (N\kappa)^{10}$ and have \begin{multline}\label{eq:Aformula}A = \frac{\kappa-1}{12} \sum_{LM=N} \frac{d(L)}{\zeta^{(M)}(2)} \sum_{\substack{k \leq x^{1/2} \\ (k,M)=1}} \frac{1}{k^2} \sum_{\substack{\ell \leq x/k^2 \\ (\ell,M)=1}} \frac{\overline{\epsilon}(\ell)}{\ell} \sum_{WQ= M} \mu(Q) W F(W,\epsilon) \prod_{p^2 \mid W} \left(1 - \frac{1}{p^2}\right) \\  \times \sum_{\substack{ q \mid Q^\infty \\ (q,W)=1}} \frac{\overline{\epsilon}(q)}{q} \Delta_{\kappa,W,\epsilon}(m,q^2\ell^2),\end{multline} and $E$ is the error term from \eqref{eq:ne3} of size \begin{equation}\label{errG0N} E\ll_{r,\eps} \kappa N^{1+\eps}x^{-\frac{\delta }{20}+\eps} + N^\eps.\end{equation} Applying \eqref{petersson2} to $A$ we get that $$A=D + OD,$$ where $D$ and $OD$ are the contributions from the diagonal term and off-diagonal term of \eqref{petersson2}, respectively. We insert $\delta_{m=q^2\ell^2} \delta_{\cond(\epsilon) \mid W}$ for $\Delta_{\kappa, W,\epsilon}(m,q^2\ell^2)$ in \eqref{eq:Aformula} to find  $$ D  =  \frac{\kappa-1}{12}  \frac{\overline{\epsilon}(m^{\frac{1}{2}})}{m^{\frac{1}{2}}}\sum_{LM=N} \frac{d(L)}{\zeta^{(M)}(2)} \sum_{\substack{k \leq x^{1/2}/m^{1/4} \\ (k,M)=1}} \frac{1}{k^2} \sum_{WQ= M} \mu(Q) W F(W,\epsilon) \prod_{p^2 \mid W} \left(1 - \frac{1}{p^2}\right)\delta_{\cond(\epsilon) \mid W} .$$
Extending the sum over $k$ to infinity we conclude that
$$D = \frac{\kappa-1}{12}  \frac{\overline{\epsilon}(m^{\frac{1}{2}})}{m^{\frac{1}{2}}}\sum_{LM=N} d(L) \sum_{WQ= M} \mu(Q) W F(W,\epsilon) \prod_{p^2 \mid W} \left(1 - \frac{1}{p^2}\right)\delta_{\cond(\epsilon) \mid W}  + O_\eps \left( \frac{\kappa N^{1+\eps}}{x^{\frac{1}{2}}m^{\frac{1}{4}}} |\epsilon(m^{\frac{1}{2}})|\right) .$$ By a tedious case check on prime powers we have $$ \psi(N)\delta_{\cond(\epsilon) \mid N} = \sum_{LM=N} M F(M,\epsilon) \delta_{\cond(\epsilon) \mid M} \prod_{p^2 \mid M}\left(1 - \frac{1}{p^2}\right).$$ 
Therefore the result of the diagonal contribution is \begin{equation}\label{diag}D= \frac{\kappa-1}{12}  \frac{\overline{\epsilon}(m^{\frac{1}{2}})}{m^{\frac{1}{2}}}\psi(N)+ O\left(\frac{\kappa N^{1+\eps}}{x^{\frac{1}{2}}m^{\frac{1}{4}}}|\epsilon(m^{\frac{1}{2}})|\right),\end{equation} which matches what one finds directly from the identity contribution of the Eichler-Selberg trace formula.

Now we treat the off-diagonal terms. Let \begin{equation}\label{ANTproblem1}B(Y,m,W) = \sum_{\substack{\ell \leq Y \\ (\ell,M)=1}} \sum_{\substack{q \mid Q^\infty \\ (q,W)=1}} \frac{\overline{\epsilon}(q\ell)}{q\ell} \sum_{c \equiv 0 \mods W} \frac{S_{\epsilon}(m,q^2\ell^2,c)}{c} J_{\kappa-1}\left(\frac{4 \pi q\ell \sqrt{m}}{c}\right).\end{equation} Then we have that $$OD = \frac{\kappa-1}{12} \sum_{LM=N} \frac{d(L)}{\zeta^{(M)}(2)}  \sum_{WQ= M} \mu(Q) W F(W,\epsilon) \prod_{p^2 \mid W} \left(1 - \frac{1}{p^2}\right) \sum_{\substack{k \leq x^{1/2} \\ (k,M)=1}} \frac{1}{k^2} B(x/k^2,m,W).$$

  \begin{lemma}\label{ILSlem}
  Let $d_3$ denote the $3$-divisor function. For any $m,n\geq 1$ we have  
   \begin{multline*}\sum_{c \equiv 0 \mods W} \frac{S_\epsilon(m,n,c)}{c} J_{\kappa-1}\left(\frac{4 \pi \sqrt{mn}}{c}\right)\\  \ll \cond(\epsilon)^{\frac{1}{4}}\prod_{p|\cond(\epsilon)}p^{\frac{1}{4}} \frac{(m,n,W)^{\frac{1}{2}} d_3((m,n))d(W)}{W\kappa^{\frac{5}{6}}} \left( \frac{mn}{\sqrt{mn} + \kappa W}\right)^{\frac{1}{2}} \log 2 mn .\end{multline*}
  \end{lemma}
  \begin{proof}
 The proof is identical to \cite[Cor.~2.2]{ILS} but with the following bound on the Kloosterman sum in lieu of the standard bound without nebentype character.   \begin{lemma}\label{Weil}  For integers $c\in N\Z$ and $a,b \in\Z$ with $c\neq 0$  and $\cond(\epsilon) \mid N$,
  we have the estimate
\[|S_{\epsilon}(a,b;c)|\le d(c)\,(a,b,c)^{\frac{1}{2}}
  \,c^{\frac{1}{2}}\,\cond(\epsilon)^{\frac{1}{4}}\cond^*(\epsilon)^{\frac{1}{4}}.\]
\end{lemma}
 \begin{proof} See Knightly and Li \cite[Thm.~9.2]{KLkuznetsov}. \end{proof}  \end{proof}

Applying Lemma \ref{ILSlem} and estimating sums by integrals we find $$B(Y,m,W) \ll \cond(\epsilon)^{\frac{1}{4}}\cond^*(\epsilon)^{\frac{1}{4}} \frac{ d(W)m^{\frac{1}{4}}Y^{\frac{1}{2}}}{W\kappa^{\frac{5}{6}}}\log mY,$$ hence one estimates that $$OD \ll_\eps \cond(\epsilon)^{\frac{1}{4}}\cond^*(\epsilon)^{\frac{1}{4}}  x^{\frac{1}{2}} \kappa^{\frac{1}{6}} m^{\frac{1}{4}}N^\eps \log mx .$$ We have $\tr(T'_m|  S_\kappa(\Gamma_0(N,\epsilon)))  =  \overline{D}+\overline{OD}+\overline{E}$, and so collecting error terms we obtain 
 \begin{multline}\label{tracewithx} \tr(T'_m|  S_\kappa(\Gamma_0(N,\epsilon))) 
 = \frac{\kappa-1}{12}  \frac{\epsilon(m^{\frac{1}{2}})}{m^{\frac{1}{2}}}\psi(N)+ O_\eps \Bigg( \cond(\epsilon)^{\frac{1}{4}}\cond^*(\epsilon)^{\frac{1}{4}}   x^{\frac{1}{2}} \kappa^{\frac{1}{6}} m^{\frac{1}{4}}N^\eps \log mx \\  + \kappa N^{1+\eps} x^{-\frac{\delta}{20}+ \eps} + N^\eps \Bigg).\end{multline}

We now optimize the value of $r$. By \cite{HLAppendix, BanksGL3}, the exponent $\delta=1$ is admissible.  
The error in \eqref{tracewithx} is minimized when $$x^\frac{11}{20} = \frac{N\kappa^{\frac{5}{6}}}{m^{\frac{1}{4}} \cond(\epsilon)^{\frac{1}{4}}\cond^*(\epsilon)^{\frac{1}{4}} }.$$ Let us assume that there is some $\eta>0$ such that \begin{equation}\label{mbound}m^{\frac{1}{4}}\cond(\epsilon)^{\frac{1}{4}}\cond^*(\epsilon)^{\frac{1}{4}}  \ll (N \kappa^{\frac{5}{6}})^{1-\eta}.\end{equation} We choose $r\geq 1$ to be the nearest integer to $$\frac{11}{2} \left(1 + \frac{ \log (\kappa^{\frac{1}{6}} m^{\frac{1}{4}}\cond(\epsilon)^{\frac{1}{4}}\cond^*(\epsilon)^{\frac{1}{4}} )}{\log (N \kappa^{\frac{5}{6}}) - \log (m^{\frac{1}{4}} \cond(\epsilon)^{\frac{1}{4}}\cond^*(\epsilon)^{\frac{1}{4}} )}\right) ,$$ which by \eqref{mbound} is then bounded above uniformly in terms of $\eta>0$ only.

\section{Analysis for $\Gamma(M,N)$}\label{AforMN}

Recall from \eqref{eq:mnstart} that $$\tr (\langle \overline{d} \rangle T_m' | S_\kappa(\Gamma(M,N)))   = \sum_{\epsilon \mods N} \overline{\epsilon}(d) \tr(T_m' | S_\kappa(\Gamma_0(MN),\epsilon)),$$ and that in section \ref{AforG0} we decomposed the interior of this as $$\overline{\tr(T_m' | S_\kappa(\Gamma_0(MN),\epsilon))} = D + OD + E.$$ Summing the formula \eqref{diag} for $D$ and \eqref{errG0N} for $E$ trivially over characters $\epsilon \mods{N}$ we get \begin{multline}\label{formula51} \tr (\langle \overline{d} \rangle T_m' | S_\kappa(\Gamma(M,N)))   = \frac{\kappa-1}{24}m^{-\frac{1}{2}} \varphi(N) \psi(NM) \left(\delta_N(m^{\frac{1}{2}} d, 1) + (-1)^\kappa \delta_N(m^{\frac{1}{2}} d, -1)\right)   \\ + \overline{OD^*}+ O_{\eta,\eps}\left(\kappa (MN^2)^{1+\eps} x^{-\frac{\delta}{20}+\eps} + N(MN)^\eps\right),\end{multline} where $x^r = (MN\kappa)^{10}$, $r$ is a parameter to be chosen later, and $$OD^*= \sum_{\substack{\epsilon \mods N \\ \epsilon(-1) = (-1)^\kappa}} \epsilon(d) OD.$$ Let \begin{multline*}B^*(Y,m,W) =  \sum_{\substack{\epsilon \mods N \\ \epsilon(-1)=(-1)^\kappa}} \epsilon(d) F(W,\epsilon) \sum_{\substack{(\ell,K) =1 \\ \ell \leq Y}}\sum_{\substack{q \mid Q^\infty \\ (q,W)=1}}\frac{\overline{\epsilon(q\ell)}}{q\ell} \sum_{c \equiv 0 \mods W} \frac{S_\epsilon(m,q^2\ell^2,c)}{c} \\ \times J_{\kappa-1}\left( \frac{4 \pi \ell q \sqrt{m}}{c}\right),\end{multline*} so that we have \begin{equation}\label{Ostar}OD^* = \frac{\kappa -1}{12}  \sum_{LK=MN} \frac{d(L)}{\zeta^{(K)}(2) } \sum_{\substack{ k \leq x^{1/2} \\ (k,K)=1}} \frac{1}{k^2} \sum_{WQ = K } \mu(Q) W  \prod_{p^2 \mid W} (1-\frac{1}{p^2}) B^*(x/k^2,m,W).\end{equation} 

We would like to utilize the orthogonality of characters over $\epsilon \mods N$.  To implement this, we now refresh the notation. Suppose $W,N\geq 1$ are integers such that $W \mid N^2$. For $a,b,d,\kappa \in \Z$ and $1\leq c \equiv 0 \mods W$ define  $$T_W(a,b,c) := \sum_{\substack{\epsilon \mods N \\ \epsilon(-1)=(-1)^\kappa \\ \cond(\epsilon) \mid W}} \epsilon(d)\overline{\epsilon}(b)  F(W,\epsilon) S_\epsilon(a,b,c).$$
With this notation, we have \begin{equation}\label{Bformula}B^*(Y,m,W) = \sum_{\substack{(\ell,K) =1 \\ \ell \leq Y}}\sum_{\substack{q \mid Q^\infty \\ (q,W)=1}}\frac{1}{q\ell} \sum_{c \equiv 0 \mods W} \frac{T_W(m,q^2 \ell^2,c)}{c} J_{\kappa-1}\left( \frac{4 \pi \ell q \sqrt{m}}{c}\right).\end{equation}
We can derive a bound on $T_W$ by appealing to the Weil bound for Kloosterman sums.
\begin{lemma}\label{Tsum}
Suppose $W,N\geq 1$ such that $W\mid N^2$, $a,b,d,\kappa \in \Z$ such that $(b,W)=1$, $(d,N)=1$, and $1 \leq c \equiv 0 \mods W$. We factor $c=c_1c_2$ with $c_1 \mid W^\infty$ and $(c_2,W)=1$. Then
$$ |T_W(a,b,c)| \leq \psi(c_1) d(c_2)(a,b,c_2)^{1/2} c_2^{1/2}.$$
\end{lemma}
\begin{proof}
Consider the sum $$T'_{W}(a,b,c) := \sum_{\substack{\epsilon \mods {N} \\ \cond(\epsilon) \mid W}} \epsilon(d)\overline{\epsilon}(b)F(W,\epsilon)S_\epsilon(a,b,c),$$ which is a minor variation of $T_W(a,b,c)$, but omitting the global condition $\epsilon(-1)=(-1)^\kappa$. We first consider the sum $T'_W$ locally, returning to $T_W$ at the end of the proof.
Let $\alpha, \beta, \gamma \geq 0$ such that $\alpha \leq \gamma$, $\alpha \leq 2\beta$, $(d,p^\beta)=(b,p^\alpha)=1$, and consider $T'_{p^\alpha}(a,b,p^\gamma)$. Let $$I(\alpha,\beta) : = \sum_{\epsilon \mods {p^\beta} } \epsilon(dx)\overline{\epsilon}(b) \delta_{\cond_p(\epsilon) \leq \alpha} \begin{cases} 1+\frac{1}{p} & \text{ if }\cond_p(\epsilon)= \alpha = 1 \\ \left(1-\frac{1}{p}\right)^{-1} & \text{ if } \cond_p(\epsilon)= \alpha \geq 2 \\ 1 & \text{ else.}\end{cases}$$ By opening the Kloosterman sum and exchanging order of summation we have \begin{equation}\label{T'intermsofI}T_{p^\alpha}'(a,b,p^\gamma) = \sums_{x \mods {p^\gamma}} e \left( \frac{ax+b\overline{x}}{p^\gamma}\right) I(\alpha, \beta).\end{equation}
Next we break into four cases: \begin{enumerate}
\item\label{case1} $ \alpha> \beta$
\item\label{case2} $0 = \alpha \leq \beta$
\item\label{case3} $1 = \alpha \leq \beta$
\item\label{case4} $2 \leq \alpha \leq \beta$. 
\end{enumerate}
Recall the orthogonality relation $$ \sum_{\epsilon \mods n } \epsilon(a)\overline{\epsilon}(b) =  \varphi(n)\delta_n(a,b)$$ and the almost-orthogonality relation (see e.g.~\cite[Section 2]{HBalmostOR}) $$ \sum_{\cond(\epsilon) = c  } \epsilon(a)\overline{\epsilon}(b) =  \sum_{\delta \mid (a-b,c)} \varphi(\delta) \mu\left( \frac{c}{\delta}\right) .$$ 
We apply these to evaluate $I(\alpha,\beta)$ in cases \eqref{case1}-\eqref{case4}. We find 
\begin{equation}\label{Iformula}I(\alpha, \beta) =  \begin{cases} \varphi(p^\beta) \delta_{p^\beta}(xd,b)& \text{ if } \alpha > \beta \\
1 & \text{ if } 0=\alpha \leq \beta \\
\varphi(p) \delta_p(xd,b) + \frac{1}{p} \sum_{\delta \mid (p,xd-b)}\varphi(\delta) \mu(p^\gamma/\delta) & \text{ if } 1 = \alpha \leq \beta \\
\varphi(p^\alpha) \delta_{p^\alpha}(xd,b) + \frac{1}{p-1} \sum_{\delta \mid (p^\alpha,xd-b)}\varphi(\delta) \mu(p^\gamma/\delta) & \text{ if } 1 = \alpha \leq \beta. \end{cases}\end{equation}
Recall that $(d,p^\beta)=1$, so that $d^{-1} \mods {p^\beta}$  (or $\mods {p^\gamma}$ in cases \eqref{case3} and \eqref{case4}) exists. Inserting \eqref{Iformula} to \eqref{T'intermsofI}, we find the following.

\textbf{Case \eqref{case1}: $\beta<\alpha$.} $$T'_{p^\alpha}(a,b,p^\gamma) = \varphi(p^\beta) \sums_{\substack{x \mods {p^\gamma} \\ x \equiv d^{-1}b \mods {p^\beta}}}  e \left( \frac{ax+b\overline{x}}{p^\gamma}\right).$$ 

\textbf{Case \eqref{case2}: $\beta \geq \alpha =0$.}  $$T'_{1}(a,b,p^\gamma) = \sums_{x \mods {p^\gamma} }  e \left( \frac{ax+b\overline{x}}{p^\gamma}\right)=  S(a,b,p^\gamma).$$

\textbf{Case \eqref{case3}: $\beta \geq \alpha =1$.} $$ T'_{p}(a,b,p^\gamma) = \varphi(p) \sums_{\substack{x \mods {p^\gamma} \\ x \equiv d^{-1}b \mods {p}}}  e \left( \frac{ax+b\overline{x}}{p^\gamma}\right) + \frac{1}{p} \sum_{\delta \mid p } \varphi(\delta) \mu(p/\delta)\sums_{\substack{x \mods {p^\gamma} \\ x \equiv d^{-1}b \mods {\delta}}}  e \left( \frac{ax+b\overline{x}}{p^\gamma}\right).$$
 
 \textbf{Case \eqref{case4}: $\beta \geq \alpha \geq 2$.} $$ T'_{p^\alpha}(a,b,p^\gamma) = \varphi(p^\alpha) \sums_{\substack{x \mods {p^\gamma} \\ x \equiv d^{-1}b \mods {p^\alpha}}}  e \left( \frac{ax+b\overline{x}}{p^\gamma}\right) + \frac{1}{p-1} \sum_{\delta \mid p^\alpha } \varphi(\delta) \mu(p^\alpha/\delta)\sums_{\substack{x \mods {p^\gamma} \\ x \equiv d^{-1}b \mods {\delta}}}  e \left( \frac{ax+b\overline{x}}{p^\gamma}\right).$$
Using the Weil bound for Kloosterman sums and trivial bounds, we find for all integers $a,b$, and non-zero integers $0\leq i \leq j$, and $(y,p)=1$ we have \begin{equation}\label{KLbound} \left| \sums_{\substack{x \mods{p^j} \\ x \equiv y \mods {p^i}}} e\left(\frac{ax+b^2\overline{x}}{p^\gamma}\right) \right| \leq \begin{cases} 
 d(p^j)(a,b^2,p^j)^{\frac{1}{2}} \sqrt{p^j} & \text{ if } i = 0 \\ p^{j-i} & \text{ else.}\end{cases}\end{equation}
Applying \eqref{KLbound} to the various cases above, we find that cases \eqref{case1}, \eqref{case3}, and \eqref{case4}, i.e. when $\alpha >0$, the bound \begin{equation}\label{T'bound1}|T'_{p^\alpha}(a,b,p^\gamma)| \leq \psi(p^{\gamma}).\end{equation} In case \eqref{case2}, i.e. when $\alpha = 0$, we have \begin{equation}\label{T'bound2}|T'_{p^\alpha}(a,b,p^\gamma) | \leq d(p^\gamma) (a,b,p^\gamma)^{1/2} p^{\gamma/2}.\end{equation} Thus the estimation of $T'_{p^\alpha}(a,b,p^\gamma)$ is finished.

Now we return to the case of $T_W(a,b,c)$. We have $$T_W(a,b,c) = \frac{1}{2} T'_W(a,b,c) + \frac{(-1)^\kappa}{2} T'_W(a,-b,c),$$ so it suffices to establish the bound stated in the lemma for $T'_W(a,b,c)$. We have that $T'_W(a,b,c)$ is twisted multiplicative, i.e. we have a factorization \begin{equation}\label{eq:prod3sub}T'_W(a,b,c) = \prod_{\substack{ p^\alpha ||W  \\ p^\gamma || c}} T'_{p^\alpha} (a \overline{cp^{-\gamma}}, b \overline{cp^{-\gamma}}, p^\gamma).\end{equation}
Bounding the left hand side of \eqref{eq:prod3sub} using \eqref{T'bound1} and \eqref{T'bound2}, we conclude the proof of the lemma.
\end{proof}

Applying Lemma \ref{Tsum} to \eqref{Bformula} we get \begin{equation*} B^*(Y,m,W)  \leq \sum_{\substack{(\ell,K) =1 \\ \ell \leq Y}}\sum_{\substack{q \mid Q^\infty \\ (q,W)=1}}\frac{1}{q\ell} \sum_{W \mid c_1 \mid W^\infty} \frac{\psi(c_1)}{c_1} \sum_{(c_2,W)=1} \frac{d(c_2) (m,q^2\ell^2,c_2)^{\frac{1}{2}}}{\sqrt{c_2}}\left|J_{\kappa-1}\left( \frac{4 \pi q\ell \sqrt{m}}{c_1c_2}\right)\right|.\end{equation*} Again following closely the proof of \cite[Cor.~2.2]{ILS} we have that $$ \sum_{(c_2,W)=1} \frac{d(c_2) (a,b^2,c_2)^{\frac{1}{2}}}{\sqrt{c_2} }\left|J_{\kappa-1}\left( \frac{4 \pi b\sqrt{a}}{c_1c_2}\right)\right| \ll \frac{d_3((a,b^2))}{\kappa^{\frac{5}{6}}\sqrt{c_1}} \left( \frac{b^2a}{b\sqrt{a} + c_1\kappa}\right)^{\frac{1}{2}} \log 2b^2a.$$ We have moreover that $$ \sum_{W \mid c_1 \mid W^\infty} \frac{1}{\sqrt{c_1}}  \frac{1}{(b\sqrt{a} + c_1\kappa)^{\frac{1}{2}}} \leq \frac{2}{W^{\frac{1}{2}}b^{\frac{1}{2}}a^{\frac{1}{4}}}.$$  These last two estimations lead to \begin{align*} B^*(Y,m,W) & \ll \frac{m^{\frac{1}{4}}}{\kappa^{\frac{5}{6}}}\frac{\psi(W)}{W^\frac{3}{2}}\sum_{\substack{(\ell,K) =1 \\ \ell \leq Y}}\sum_{\substack{q \mid Q^\infty \\ (q,W)=1}}\frac{d_3((m,q^2\ell^2))}{\sqrt{q\ell}}\log (2 q^2\ell^2 m)  \\
&  \ll \frac{m^{\frac{1}{4}} Y^{\frac{1}{2}}(\log Y)^3 \log 2m}{\kappa^{\frac{5}{6}} } \frac{\psi(W)}{W^\frac{3}{2}}.\end{align*} Inserting this into \eqref{Ostar} we get $$OD^* \ll \kappa^{\frac{1}{6}} MN^{\frac{1}{2}+\eps}x^{\frac{1}{2}}(\log x)^3 m^{\frac{1}{4}} \log 2m,$$ and inserting this into \eqref{formula51} we conclude that \begin{multline}\label{GMN1} \tr (\langle \overline{d} \rangle T_m' | S_\kappa(\Gamma(M,N)))   = \frac{\kappa-1}{24}m^{-\frac{1}{2}} \varphi(N) \psi(NM) \left(\delta_N(m^{\frac{1}{2}} d, 1) + (-1)^\kappa \delta_N(m^{\frac{1}{2}} d, -1)\right)  \\ + O_{\eta,\eps}\left(\kappa^{\frac{1}{6}} MN^{\frac{1}{2}} x^{\frac{1}{2}+\eps} m^{\frac{1}{4}}  \log 2m + \kappa (MN^2)^{1+\eps} x^{-\frac{\delta}{20}+\eps} + N(MN)^\eps\right).\end{multline}
Now we optimize the value of $r$. The error term is minimized when $$x^\frac{11}{20} = \frac{ N^{\frac{3}{2}}\kappa^{\frac{5}{6}}}{m^{\frac{1}{4}}}.$$ Let us assume that there is some $\eta>0$ such that $$m^{\frac{1}{4}} \ll (N^{\frac{3}{2}} \kappa^{\frac{5}{6}})^{1-\eta}.$$ We choose $r\geq 1$ to be the nearest integer to $$\frac{11}{2} \left( \frac{\log MN\kappa }{\log N^{\frac{3}{2}} \kappa^{\frac{5}{6}} - \log m^{\frac{1}{4}}}\right) ,$$ which is then bounded above uniformly in terms of $\eta>0$ only.

\section{Proof of Proposition \ref{prop:Holder}}\label{holderproof}
\begin{proof}  
We have by Lemma \ref{Rankin} that $$A[\alpha_f] = \frac{\kappa-1}{4 \pi} \frac{\Vol X_0(N)}{\zeta^{(N)}(2)} \sum_{f \in H^\star_\kappa(N,\epsilon)} \omega_f \alpha_f L^{(N)}(1,\Ad^2 f)  = \frac{\kappa-1}{4 \pi} \frac{\Vol X_0(N)}{\zeta^{(N)}(2)} A^h[\alpha_f L^{(N)}(1,\Ad^2 f)].$$ Recall that we have set $\rho_f(n)$ to be the Dirichlet series coefficients of $L^{(N)}(s,\Ad^2 f)$, along with $$\omega_f(x) = \sum_{n\leq x} \frac{\rho_f(n)}{n}, \,\,\,\text{ and }\,\,\,\omega_f(x,y)= \sum_{x< n\leq y} \frac{\rho_f(n)}{n}.$$ 
\begin{lemma}\label{AFE}
We have $$L^{(N)}(1, \Ad^2 f) = \omega_f(x) + \omega_f(x,y) + O_\eps((N\kappa)^{\frac{1}{2}} y^{-\frac{1}{2}+ \eps}).$$
Assuming the generalized Lindel\"of hypothesis, the $(N\kappa)^{1/2}$ can be reduced to $(N\kappa)^\eps$.
\end{lemma}
\begin{proof}[Proof (sketch)]
For $c,T,y>0$, we apply Perron's formula (see e.g. \cite[pg. 105]{Davenport}) to calculate $\omega_f(y)$, finding $$ \omega_f(y) = \frac{1}{2\pi i } \int_{c-iT}^{c+iT} L^{(N)}(1+s,\Ad^2 f)\frac{y^s}{s}\,ds + O \left(y^c \sum_{n\geq 1} \frac{\rho_f(n)}{n^{1+c}}\min ( 1, T^{-1} | \log y/n|^{-1})\right).$$ We shift the contour to $\real(s)=-2$ to get \begin{multline}\label{contourshift} \omega_f(y) = L^{(N)}(1,\Ad^2 f) + \frac{1}{2\pi i }\left(\int_{c-iT}^{-2-iT}  +  \int_{-2-iT}^{-2+iT} + \int_{-2+iT}^{c+iT}\right) L^{(N)}(1+s,\Ad^2 f)\frac{y^s}{s}\,ds \\ + O \left(y^c \sum_{n\geq 1} \frac{\rho_f(n)}{n^{1+c}}\min ( 1, T^{-1} | \log y/n|^{-1})\right).\end{multline}
By an inspection of the functional equation for $L(s, f \otimes \overline{f})$ found in \cite[Example 1]{LiLseriesRankinType}, we have the convexity bound (see e.g. \cite[(5.20)]{IK}) \begin{equation}\label{convexity}L^{(N)}(s,\Ad^2 f) \ll [ (\kappa N)^2 (1+|t|)^3]^{\frac{1-\sigma}{2}+\eps},\end{equation} where $s = \sigma +it$, valid for $\sigma \leq 1$.  Choosing $c=\eps$, $T= (N\kappa)^{-\frac{1}{2}} y^{\frac{1}{2}+\eps}$, and estimating all of the terms in \eqref{contourshift} directly, one finds the estimate in the statement of the Lemma. 

If one assumes the generalized Lindel\"of hypothesis in place of \eqref{convexity}, then we shift the contour to $\real(s) =-1/2$ instead of $-2$ and follow the same steps. 
\end{proof}

By Lemma \ref{AFE} we have \begin{equation}\label{Adecomp}A[\alpha_f] = \frac{\kappa-1}{4 \pi} \frac{\Vol X_0(N)}{\zeta^{(N)}(2)} \left( A^h[\omega_f(x) \alpha_f] + A^h[\omega_f(x,y) \alpha_f] + O((N\kappa)^{\frac{1}{2}}y^{-\frac{1}{2} + \eps} A^h[\left| \alpha_f \right|])\right).\end{equation} By the hypothesis \eqref{eq:HkNeHypothesis1} we have $A^h[|\alpha_f| ] \ll_\eps (N\kappa)^\eps$, and so taking $y=(N \kappa)^{3+\eps}$, we find that the $O$ term in \eqref{Adecomp} is $\ll (N\kappa)^{-1}$.

Next we consider the second term and treat it using the following large sieve inequality.  This is a slight variation on Corollary 6 of \cite{DukeKowalski}, see also \cite[Prop.~1]{KMrankofJ0}. Let $\lambda^{(2)}_f(n)$ be the Dirichlet series coefficients of the automorphic adjoint-square $L$-function $L(s,\Ad^2 \pi)$, where $f$ is a newform for the representation $\pi$.  If $(n,N)=1$ then we have that $\lambda^{(2)}_f(n) = \rho_f(n).$ \begin{proposition}\label{prop:largesieve}Let $X\geq (N\kappa)^8$.  We have for all $\eps>0$ that \begin{equation} \sum_{f \in H^\star_{\kappa}(N,\epsilon)} \left| \sum_{n \leq X} a_n \lambda_f^{(2)}(n)\right|^2 \ll_\eps X^{1+\eps} \sum_{n \leq X} \left| a_n \right|^2\end{equation} for any finite family $(a_n)_{1\leq n \leq X}$ of complex numbers, where the constant depends only on $\eps$. \end{proposition}
By following closely Kowalski and Michel \cite[\S 3.3]{KMrankofJ0} one deduces from Proposition \ref{prop:largesieve} the following Lemma. 
\begin{lemma}\label{lem3}
Let $r\geq 1$ be an integer such that $x^r \geq (N \kappa)^{10}$.  Then for all $\eps >0$ we have $$A[\omega_f(x,y)^{2r}] \ll_{r,\eps} (N\kappa)^\eps,$$ where the implied constant depends only on $r$ and $\eps$.  \end{lemma} \begin{proof}It suffices to replace instances of $\lambda_f(n^2)$ in \cite[Lemma 3]{KMrankofJ0} by $\overline{\epsilon}(n) \lambda_f(n^2)$ and to use equations \eqref{eq:rho} and \eqref{eq:rhoinverse} in the place of the equations (15) and (16) of Kowalski and Michel.\end{proof}
We now can give an estimate for the second term of \eqref{Adecomp}.  We use H\"older's inequality to separate $\omega_f(x,y)$ from $\alpha_f$, and Lemma \ref{lem3} to handle the former.  Precisely, let $s$ be defined by $(2r)^{-1}+s^{-1}=1$.  Applying H\"older's inequality we find for any integer $r\geq 1$ that \begin{align*} A^h[\omega_f(x,y)\alpha_f ] & =  \sum_{f \in H^\star_\kappa} \omega_f \omega_f(x,y)\alpha_f \\ & \leq A[\omega_f(x,y)^{2r}]^{\frac{1}{2r}} \left( \sum_{f \in H^\star_\kappa(N,\epsilon)}  \left( \omega_f |\alpha_f|\right)^s \right)^{\frac{1}{s}} \\ & \leq A^{\frac{1}{2r}} A[\omega_f(x,y)^{2r}]^{\frac{1}{2r}}  A^h[|\alpha_f|]^{\frac{1}{s}},\end{align*} where $$A =\max_{f \in H^\star_\kappa(N,\epsilon)} \omega_f |\alpha_f| \ll_\eps (N \kappa)^{-\delta + \eps}$$ by hypothesis \eqref{eq:HkNeHypothesis2}.  Suppose now that $r$ is sufficiently large so that $x^{r}\geq (N\kappa)^{10}$.  Then Lemma \ref{lem3} applies, and we have $$ A[\omega_f(x,y)^{2r}]^{\frac{1}{2r}} \ll_{r,\eps} (N\kappa)^\eps.$$ Lastly, by hypothesis \eqref{eq:HkNeHypothesis1} we have $$ A^h[|\alpha_f|]^{\frac{1}{s}} \ll_\eps (N\kappa)^\eps.$$ Putting these estimates together, we find that $A^h[\omega_f(x,y)\alpha_f] \ll_{r,\eps} (N\kappa)^{-\frac{\delta}{2r} + \eps},$ and so derive the bound claimed in Proposition \ref{prop:Holder}.
\end{proof}

\section{Proof of Theorem \ref{thm:petersson}}\label{petproof}
\begin{proof}The strategy of the proof is the pick an orthogonal basis for $S_\kappa(\Gamma_0(N),\epsilon)$ and compute the Fourier coefficients of basis elements explicitly.  For $f$ a modular function of weight $\kappa$, we denote by $f_{\vert d}(z)=d^{\frac{\kappa}{2}}f(dz)$.  Atkin-Lehner theory gives an orthogonal direct sum decomposition $$S_\kappa(\Gamma_0(N),\epsilon) = \bigoplus_{LM=N} \bigoplus_{f \in H^\star_{\kappa}(M,\epsilon)} S_\kappa(L;f,\epsilon),$$ where $S_\kappa(L;f,\epsilon)= \operatorname{span}\{f_{\vert \ell}: \ell \mid L\}$ is called an oldclass.  Note that the inner sum is $\{0\}$ unless $\cond(\epsilon) \mid M$, so we may assume this for the remainder of the proof.

To pick an orthogonal basis for $S_\kappa(\Gamma_0(N),\epsilon)$ it then suffices to pick a orthonormal basis for each oldclass $S_\kappa(L;f,\epsilon)$.  We use a basis for the oldclasses first due to Schulze-Pillot/Yenirce \cite[Thm.~8]{SPYbasis}.  The basis constructed by Schulze-Pillot/Yenirce is the same as the one found by Rouymi \cite{Rouymi} in the case of prime power level and trivial nebentypus and Ng \cite{NgBasis} in the case of arbitrary level and trivial nebentypus, see also Blomer and Mili\'cevi\'c \cite[Ch.~5]{BlomerMilicevic2ndMoment} and Humphries \cite[Lemma 3.15]{HumphriesBasis}.  Each of these preceding works used the Rankin-Selberg method to compute inner products and orthonomalize the oldclasses. Schulze-Pillot/Yenirce however took a different and simpler path, using the trace operator to compute the inner products.

Let $f \in H^\star(M,\epsilon)$.  For integers $d \mid g$ one defines a joint multiplicative function $\xi_g(d)$. On prime powers $\xi_g(d)$ is given for $\nu \geq 2$ as follows: \begin{align*} \xi_1(1) & = 1,  & \xi_{p^\nu}(p^\nu)  & = \left( 1- \frac{|\lambda_f(p)|^2}{p(1+\frac{\eps_{0,M}(p)}{p})^2}\right)^{-\frac{1}{2}}\left(1-\frac{\eps_{0,M}(p)^2}{p^2}\right)^{-\frac{1}{2}}, \\ 
\xi_p(p) & =  \left(1- \frac{|\lambda_f(p)|^2}{p(1+\frac{\epsilon_{0,M}(p)}{p})^2}\right)^{-\frac{1}{2}}, & \xi_{p^\nu}(p^{\nu-1}) & = \frac{-\overline{\lambda_f(p)}}{\sqrt{p}}\xi_{p^\nu}(p^{\nu}), \\ 
\xi_p(1) & = \frac{-\overline{\lambda_f(p)}}{\sqrt{p}(1+ \epsilon_{0,M}(p)/p)}\xi_p(p), & \xi_{p^{\nu}}(p^{\nu-2}) & = \frac{\overline{\epsilon(p)}}{p}\xi_{p^\nu}(p^{\nu}), \end{align*} and $\xi_{p^a}(p^b)= 0$ in all other cases.  
\begin{proposition}[Thm.~9 \cite{SPYbasis}]\label{prop:onbasis}
Let $M \mid N$ and let $f \in H^\star_\kappa(M,\epsilon)$.  The set of functions $$ \{ f^{(g)}(z) = \sum_{d \mid g} \xi_g(d) d^{\frac{\kappa}{2}}f(dz) : g \mid L\}$$ is an orthogonal basis for $S_k(L;f,\epsilon)$.  In fact, if $f$ is $L^2(\Gamma_0(N)\backslash \mathcal{H})$-normalized, then the above set is in fact orthonormal.  
\end{proposition}

Now that we have an orthonormal basis for $S_\kappa(\Gamma_0(N),\epsilon)$, we follow Barrett, Burkhardt, DeWitt, Dorward, and Miller \cite{BBDDM} to derive the Petersson formula for newforms Theorem \ref{thm:petersson}.  

Let $f\in H_\kappa^\star(M,\epsilon)$ have Fourier coefficients $a_f(n)$ and be normalized so that $a_f(1)=1$.  Of course $f(z)/||f||_N$ is $L^2(\Gamma_0(N)\backslash \mathcal{H})$-normalized, so using the basis in Proposition \ref{prop:onbasis} we have \begin{align}\Delta_{\kappa,N,\epsilon}(m,n) & = \frac{c_\kappa}{(mn)^{\frac{\kappa-1}{2}}} \sum_{g \in \mathcal{B}_\kappa(\Gamma_0(N),\epsilon)} b_g(n)\overline{b_g(m)} \nonumber \\
& =   \frac{c_\kappa}{( mn)^{\frac{\kappa-1}{2}}} \sum_{LM=N} \sum_{f \in H_\kappa^\star(m,\epsilon)} \frac{1}{\langle f, f \rangle_N}\sum_{g \mid L} \overline{a_{f^{(g)}}(m)}a_{f^{(g)}}(n).  \label{applybasis} \end{align} 
By definition of $f^{(g)}$ we have $$a_{f^{(g)}}(n) = \sum_{d \mid (g,n)} \xi_g(d)d^{\frac{\kappa}{2}}a_f(\frac{n}{d}),$$ which are now expressible in terms of Hecke eigenvalues $\lambda_f(n)$ normalized so that $|\lambda_f(n)|\leq d(n)$. We have then that \begin{align*} \Delta_{\kappa,N,\epsilon}(m,n) & = \frac{c_\kappa}{(mn)^{\frac{\kappa-1}{2}}}  \sum_{LM=N} \sum_{f \in H_\kappa^\star(M,\epsilon)} \frac{1}{||f||_N^2}\sum_{g \mid L}  \overline{\left( \sum_{d \mid (g,m)} \xi_g(d)d^{\frac{\kappa}{2}}a_f(\frac{m}{d}) \right)} \left( \sum_{d \mid (g,n)} \xi_g(d)d^{\frac{\kappa}{2}}a_f(\frac{n}{d})\right) \\ 
& = c_\kappa \sum_{LM=N} \sum_{f \in H_\kappa^\star(M,\epsilon)} \frac{1}{||f||_N^2}\sum_{g \mid L}  \overline{\left( \sum_{d \mid (g,m)} \xi_g(d)d^{\frac{1}{2}}\lambda_f(\frac{m}{d}) \right)} \left( \sum_{d \mid (g,n)} \xi_g(d)d^{\frac{1}{2}}\lambda_f(\frac{n}{d})\right) \\
& = c_\kappa \sum_{LM=N} \sum_{f \in H_\kappa^\star(N,\epsilon)}   \frac{1}{||f||_N^2}\sum_{g \mid L} \Xi_{g}(m,n,f),\end{align*} where we have set $$ \Xi_g(m,n,f) = \overline{\left( \sum_{d \mid (g,m)} \xi_g(d)d^{\frac{1}{2}}\lambda_f(\frac{m}{d}) \right)} \left( \sum_{d \mid (g,n)} \xi_g(d)d^{\frac{1}{2}}\lambda_f(\frac{n}{d})\right)$$ for $g \mid L \mid N$.

Now suppose that $(d_1,d_2)=1$ and $d_1d_2 \mid m$.  Then by Hecke multiplicativity we have $$\lambda_f(\frac{m}{d_1}) \lambda_f(\frac{m}{d_2}) = \lambda_f(m) \lambda_f(\frac{m}{d_1d_2}),$$ so that for $(g_1,g_2)=1$ we have $$\Xi_{g_1}(m,n,f) \Xi_{g_2}(m,n,f) = \overline{\lambda_f(m)}\lambda_f(n)\Xi_{g_1g_2}(m,n,f).$$   
Therefore $$\Delta_{\kappa,N,\epsilon}(m,n)= c_\kappa \sum_{LM=N} \sum_{f \in H_\kappa^\star(M,\epsilon)}   \frac{1}{||f||_N^2}\left(\overline{\lambda_f(m)}\lambda_f(n)\right)^{1-\omega(L)} \prod_{p^\alpha || L } \left(\sum_{d \mid p^\alpha} \Xi_{d}(m,n,f)\right),$$ where $\omega(n)$ is the number of distinct prime factors of $n$.  Let $$V_{p^\alpha} (m,n,f) = \sum_{d \mid p^\alpha}\Xi_d(m,n,f) = (1 * \Xi)_{p^\alpha}(m,n,f),$$ where $*$ denotes Dirichlet convolution.  We suppose now that $(m,n,N)=1$ and calculate.   
\begin{lemma}[\cite{BBDDM} Appendix A]\label{BBDDMappA}
If $(m,n,N)=1$ then we have  \begin{equation*} \begin{split} V_{p^\alpha} (m,n,f) = & \overline{\lambda_f(m)} \lambda_f(n) \left( 1+ |\xi_p(1)|^2+ |\xi_{p^2}(1)|^2\right) \\ & + \delta_{p \mid m} \overline{\lambda_f(m/p)}\lambda_f(n) p^{\frac{1}{2}} \left( \overline{\xi_p(p)} \xi_p(1) + \overline{\xi_{p^2}(p)} \xi_{p^2}(1)\right)\\ & + \delta_{p \mid n} \overline{\lambda_f(m)}\lambda_f(n/p) p^{\frac{1}{2}} \left( \overline{\xi_p(1)} \xi_p(p) + \overline{\xi_{p^2}(1)} \xi_{p^2}(p)\right) \\ & + \delta_{p^2 \mid m } \overline{\lambda_f(m/p^2)} \lambda_f(n) p \overline{\xi_{p^2}(p^2)} \xi_{p^2}(1) +   \delta_{p^2 \mid n } \overline{\lambda_f(m)} \lambda_f(n/p^2) p \overline{\xi_{p^2}(1)} \xi_{p^2}(p^2), \end{split}\end{equation*} if $\alpha\geq 2$ and  \begin{equation*} \begin{split} V_{p^\alpha} (m,n,f) = & \overline{\lambda_f(m)} \lambda_f(n) \left( 1+ |\xi_p(1)|^2\right) \\ & + \delta_{p \mid m} \overline{\lambda_f(m/p)}\lambda_f(n) p^{\frac{1}{2}} \overline{\xi_p(p)} \xi_p(1) \\ & + \delta_{p \mid n} \overline{\lambda_f(m)}\lambda_f(n/p) p^{\frac{1}{2}}  \overline{\xi_p(1)} \xi_p(p)  
,\end{split}\end{equation*} if $\alpha =1$.
\end{lemma}
\begin{proof}
We actually have if $\alpha\geq 2$ that $$V_{p^\alpha}(m,n,f) = \Xi_1(m,n,f) + \Xi_p(m,n,f) + \Xi_{p^2}(m,n,f).$$  The other summands vanish because by our assumption $(m,n,N)=1$, since if $p\mid m$ then $p \nmid n$ because $p \mid N$.  So each $p$ divides either $m$ or $n$ but never both.  Then, we have that $\xi_{p^\beta}(1)=0$ for $\beta\geq 3$.  In fact, even more terms vanish.  We have \begin{equation*} \begin{split} V_{p^\alpha} (m,n,f) = & \overline{\lambda_f(m)} \lambda_f(n) \left( 1+ |\xi_p(1)|^2+ |\xi_{p^2}(1)|^2\right) \\ & + \delta_{p \mid m} \overline{\lambda_f(m/p)}\lambda_f(n) p^{\frac{1}{2}} \left( \overline{\xi_p(p)} \xi_p(1) + \overline{\xi_{p^2}(p)} \xi_{p^2}(1)\right)\\ & + \delta_{p \mid n} \overline{\lambda_f(m)}\lambda_f(n/p) p^{\frac{1}{2}} \left( \overline{\xi_p(1)} \xi_p(p) + \overline{\xi_{p^2}(1)} \xi_{p^2}(p)\right) \\ & + \delta_{p^2 \mid m } \overline{\lambda_f(m/p^2)} \lambda_f(n) p \overline{\xi_{p^2}(p^2)} \xi_{p^2}(1) +   \delta_{p^2 \mid n } \overline{\lambda_f(m)} \lambda_f(n/p^2) p \overline{\xi_{p^2}(1)} \xi_{p^2}(p^2) .\end{split}\end{equation*}  Inserting the formulas for $\xi$, we complete the proof.  The formula for the $\alpha=1$ case is even simpler as we can drop the $p^2$ terms.  
\end{proof}

Recall we write $ML=N$ and $f \in H^\star_\kappa(M,\epsilon)$.
\begin{lemma}
If $(m,N)=1$ and $(n,N)=1$ then we have $$\left(\overline{\lambda_f(m)}\lambda_f(n)\right)^{1-\omega(L)} \prod_{p^\alpha || L } V_{p^\alpha}(m,n,f) = \overline{\lambda_f(m)}\lambda_f(n) \prod_{p || L } \left( 1+ |\xi_{p}(1)|^2 \right) \prod_{p^2 | L} \left( 1 + | \xi_p(1)|^2 +|\xi_{p^2}(1)|^2\right).$$
\end{lemma}
\begin{proof}
Note that the conditions $(m,N)=1$ and $(n,N)=1$ imply that $p\nmid m$ and $p \nmid n$.  So the formula above follows immediately from the formulas in Lemma \ref{BBDDMappA}.
\end{proof}

One has that $||f||_N^2 = \frac{\psi(N)}{\psi(M)}||f||_M^2$ since $f\in H_\kappa^\star(M,\epsilon)$. Thus \begin{multline*} \Delta_{\kappa,N,\epsilon}(m,n)= c_\kappa \sum_{LM=N}\frac{\psi(M)}{\psi(N)} \sum_{f \in H_\kappa^\star(M,\epsilon)}   \frac{1}{||f||_M^2} \overline{\lambda_f(m)}\lambda_f(n) \prod_{p || L } \left( 1+ |\xi_{p}(1)|^2 \right)\\ \times  \prod_{p^2 | L} \left( 1 + | \xi_p(1)|^2 +|\xi_{p^2}(1)|^2\right).\end{multline*}

Next we insert the definitions of the $\xi$ functions.  Let $$r_f(p) = 1- \frac{ |\lambda_f(p)|^2}{p(1+ \frac{\epsilon_{0,M}(p)}{p})^2}, $$ so $$ r_f(p)^{-1} = 1+ \frac{ |\lambda_f(p)|^2}{p(1+ \frac{\epsilon_{0,M}(p)}{p})^2}  + \left( \frac{ |\lambda_f(p)|^2}{p(1+ \frac{\epsilon_{0,M}(p)}{p})^2}\right)^2 + \cdots,$$ where $\epsilon_{0,M}$ denotes the trivial character modulo $M$.  Observe that $$1+|\xi_p(1)|^2 = r_f(p)^{-1}$$ and $$ 1+ |\xi_p(1)|^2+ |\xi_{p^2}(1)|^2 = r_f(p)^{-1} \left(1- \frac{\epsilon_{0,M}(p)}{p^2}\right)^{-1}.$$

Then we get $$\Delta_{\kappa,N,\epsilon}(m,n)= c_\kappa \sum_{LM=N}\frac{\psi(M)}{\psi(N)}\prod_{p^2 \mid L}  \left(1- \frac{\epsilon_{0,M}(p)}{p^2}\right)^{-1}\sum_{f \in H_\kappa^\star(M,\epsilon)}   \frac{\overline{\lambda_f(m)}\lambda_f(n)}{||f||_M^2}  \prod_{p | L } \frac{1}{r_f(p)}.$$

Next we need a formula for $r_f(p)^{-1}$. 
Recall from \eqref{Adformula} that at a prime $p \nmid M$ that the local adjoint square $L$ function is given by \begin{equation*}\begin{split} L_p(1,\Ad^2 f) =  \frac{1}{1-p^{-2}} \sum_{\alpha \geq 0} \frac{\overline{\epsilon}(p^\alpha) \lambda_f(p^{2\alpha})}{p^{\alpha}} = \frac{1}{\left( 1-\frac{\alpha(p)/\beta(p)}{p}\right) \left( 1-\frac{1}{p}\right) \left( 1-\frac{\beta(p)/\alpha(p)}{p}\right) } \end{split}\end{equation*} so that \begin{equation*}\begin{split}  \sum_{\alpha \geq 0} \frac{\overline{\epsilon}(p^\alpha) \lambda_f(p^{2\alpha})}{p^{\alpha}} = & \frac{1 + \frac{1}{p}}{\left( 1-\frac{\alpha(p)/\beta(p)}{p}\right)  \left( 1-\frac{\beta(p)/\alpha(p)}{p}\right) } \\ 
= & \frac{1 + \frac{1}{p}}{\left(1+ \frac{1}{p}\right)^2 -\frac{|\lambda_f(p)|^2}{p} } \\ 
= & \frac{1}{\left(1+\frac{1}{p}\right)r_f(p)},\end{split}\end{equation*}
where the second equals sign follows from the formulas $$|\lambda_f(p)|^2 = \overline{\epsilon}(p) \lambda_f(p)^2,\,\,\, \lambda_f(p) =\alpha(p)+\beta(p),\,\,\, \alpha(p)\beta(p) = \epsilon(p)$$ which are valid when $p \nmid M$.  We can summarize the above calculation and Lemma \ref{Ogg} as: 
$$r_f(p)^{-1} = \begin{cases} \left(1+ \frac{1}{p}\right) \sum_{\alpha \geq 0} \frac{\overline{\epsilon}(p^\alpha) \lambda_f(p^{2\alpha})}{p^{\alpha}} & \text{ if } p \nmid M \\ 
 \left( 1- \frac{a_{M,\epsilon}(p)}{p}\right)^{-1} & \text{ if } p \mid M.
\end{cases}$$

Let $$\Delta_{\kappa,N,\epsilon}^\star(m,n) = c_\kappa \sum_{f \in H_\kappa^\star(N,\epsilon)}   \frac{\overline{\lambda_f(m)}\lambda_f(n)}{||f||_N^2}.$$ Recall the definition of $R(M,L,\epsilon)$ from the statement of Theorem \ref{thm:petersson}, which we rearrange to $$R(M,L,\epsilon) = \frac{\psi(M)}{\psi(ML)}\prod_{\substack{p^2 \mid L \\ p \nmid M}} \left(1- \frac{1}{p^2}\right)^{-1}\prod_{\substack{p \mid L \\ p \nmid M}} \left(1+ \frac{1}{p}\right) \prod_{p \mid (M,L)} \left( 1- \frac{a_{M,\epsilon}(p)}{p}\right)^{-1}.$$ 
We have then that \begin{equation}\label{eq:Delta}\Delta_{\kappa,N,\epsilon}(m,n)=  \sum_{LM=N}R(M,L,\epsilon)\sum_{\substack{\ell \mid L^\infty \\ (\ell,M)=1}} \frac{\overline{\epsilon}(\ell)}{\ell} \Delta_{\kappa,M,\epsilon}^\star(m,n\ell^2) . \end{equation} This is analogous to the first half of \cite[Prop.~4.1]{BBDDM}.  Now we would like to invert this formula, and we prepare for this with two lemmas.  
\begin{lemma}\label{lem:R}
Let $\alpha,\beta\geq 0$ and $ 0 \leq \gamma \leq \beta$ and $\cond_p (\epsilon) \leq \beta - 1$.  Then \begin{equation}\label{Req} R(p^\beta,p^\alpha,\epsilon)R(p^\gamma,p^{\beta-\gamma},\epsilon) = R(p^\gamma,p^{\alpha+\beta-\gamma},\epsilon).\end{equation}
\end{lemma}
\begin{proof}
We check cases.  

\textbf{Case $\alpha\geq 0$ and $\beta = \gamma$.} Note that $R(p^\gamma,1,\epsilon)=1$ for any $\gamma \geq 0$.

\textbf{Case $\alpha = 0$.} Note that $R(p^\beta,1,\epsilon)=1$ for any $\beta\geq 0$.

\textbf{Case $\alpha \geq 1$, $\beta =1$ and $\gamma=0$.} We have by hypothesis $\cond_p(\epsilon) = 0$, so $$R(p, p^\alpha) R(1,p) = \frac{\psi(p)}{\psi(p^{\alpha+1})} \left( 1-\frac{1}{p^2}\right)^{-1} \frac{1}{\psi(p)}\left(1+ \frac{1}{p}\right).$$On the other hand, we also have $$R(1,p^{\alpha+1})  = \frac{1}{\psi(p^{\alpha+1})} \left(1-\frac{1}{p^2} \right)^{-1}\left( 1 + \frac{1}{p}\right).$$

\textbf{Case $\alpha \geq 1$, $\beta \geq 2$ and $\gamma=0$.} We have $p \mid (p^\beta,p^\alpha)$ and $a_{p^\beta,\epsilon}(p) = 0$, so $R(p^\beta,p^\alpha,\epsilon)=p^{-\alpha}$ and 
$$R(1,p^\beta,\epsilon) = \frac{1}{\psi(p^\beta)} \left( 1- \frac{1}{p^2}\right)^{-1} \left( 1+ \frac{1}{p}\right),$$ and $$ R(1,p^{\alpha+\beta},\epsilon) = \frac{1}{\psi(p^{\alpha+\beta})} \left( 1- \frac{1}{p^2}\right)^{-1} \left( 1+ \frac{1}{p}\right).$$

\textbf{Generic case $\alpha \geq 1$, $\beta \geq 2$, $1 \leq \gamma \leq \beta-1$, and $\cond_p(\epsilon)\leq \beta-1$.} We have $$R(p^\beta,p^\alpha, \epsilon) = \frac{\psi(p^{\beta})}{\psi(p^{\alpha+ \beta})}$$ $$ R(p^\gamma, p^{\beta-\alpha},\epsilon) = \frac{\psi(p^{\gamma})}{\psi(p^{\beta})}  \left( 1- \frac{a_{p^\gamma,\epsilon}(p)}{p}\right)^{-1}$$ and  $$ R(p^\gamma, p^{\beta+\alpha- \gamma},\epsilon) = \frac{\psi(p^{\gamma})}{\psi(p^{\alpha +\beta})}  \left( 1- \frac{a_{p^\gamma,\epsilon}(p)}{p}\right)^{-1}.$$

The above cover all the cases in the lemma.
\end{proof}
\begin{lemma}\label{lem:inversionhelper}
Let $N\in \N$, $N=LM$, and $M= WQ$.  Then $$R(M,L,\epsilon)R(W,Q,\epsilon) \delta_{\cond(\epsilon) \mid W}  
= R(W,LQ,\epsilon) \delta_{\cond(\epsilon) \mid W} 
.$$
\end{lemma}
\begin{proof}
Both sides of the desired formula are multiplicative.  Let $\alpha = v_p(L)$, $\beta= v_p (M)$, and $\gamma = v_p(W).$ It then suffices to check that \begin{equation}\label{Rpeq}R(p^\beta,p^\alpha,\epsilon) R(p^\gamma,p^{\beta-\gamma},\epsilon) \delta_{\gamma \geq \cond_p (\epsilon)} = R(p^\gamma,p^{\alpha + \beta-\gamma},\epsilon)\delta_{\gamma \geq \cond_p (\epsilon)}.\end{equation} If $\cond_p (\epsilon) \leq \beta - 1$ then \eqref{Rpeq} is true by Lemma \ref{lem:R}.  So, suppose not.  Then $\beta \leq \cond_p (\epsilon ) \leq \gamma$, but $W \mid M$ so $\gamma \leq \beta$ and so $\beta = \gamma$.  In the case $\beta = \gamma$ the equation \eqref{Rpeq} is true because $R(p^\beta,1,\epsilon) = 1$.  
\end{proof}

We are now prepared to invert \eqref{eq:Delta} using Lemma \ref{lem:inversionhelper}.
We calculate \begin{equation*}\begin{split} & \sum_{LM= N} \mu(L) R(M,L,\epsilon) \sum_{\substack{\ell \mid L^\infty \\ (\ell,M) = 1}} \frac{\overline{\epsilon}(\ell)}{\ell} \Delta_{\kappa,M,\epsilon}(m,n\ell^2) \\
& =  \sum_{LM= N} \mu(L) R(M,L,\epsilon) \sum_{\substack{\ell \mid L^\infty \\ (\ell,M) = 1}} \frac{\overline{\epsilon}(\ell)}{\ell} \sum_{QW=M} R(W,Q,\epsilon) \sum_{\substack{q \mid Q^\infty \\ (q,W) = 1}} \frac{\overline{\epsilon}(q)}{q} \Delta_{\kappa,W,\epsilon}^\star(m,n\ell^2q^2) \\ 
& =  \sum_{LM=N} \mu(L) \sum_{QW = M} R(M,L,\epsilon) R(W,Q,\epsilon) \sum_{\substack{b \mid (LQ)^\infty \\ (b,W) = 1}} \frac{\overline{\epsilon}(b)}{b} \Delta_{\kappa,W,\epsilon}^\star(m,nb^2) \\ 
& =  \sum_{WX = N} R(W,X, \epsilon) \sum_{\substack{b \mid X^\infty \\ (b,W) = 1}} \frac{\overline{\epsilon}(b)}{b} \Delta_{\kappa,W,\epsilon}^\star(m,nb^2) \sum_{LQ = X} \mu(L) \\
& =  R(N,1,\epsilon) \Delta_{\kappa, N,\epsilon}^\star(m,n) \\ 
& =  \Delta_{\kappa, N,\epsilon}^\star(m,n).
\end{split}\end{equation*}
where the first equals sign is \eqref{eq:Delta}, the third is by Lemma \ref{lem:inversionhelper}, and the fourth is Mobius inversion.  
\end{proof}

\newcommand{\etalchar}[1]{$^{#1}$}
\def\Dbar{\leavevmode\lower.6ex\hbox to 0pt{\hskip-.23ex \accent"16\hss}D}
  \def\cftil#1{\ifmmode\setbox7\hbox{$\accent"5E#1$}\else
  \setbox7\hbox{\accent"5E#1}\penalty 10000\relax\fi\raise 1\ht7
  \hbox{\lower1.15ex\hbox to 1\wd7{\hss\accent"7E\hss}}\penalty 10000
  \hskip-1\wd7\penalty 10000\box7}
  \def\cfudot#1{\ifmmode\setbox7\hbox{$\accent"5E#1$}\else
  \setbox7\hbox{\accent"5E#1}\penalty 10000\relax\fi\raise 1\ht7
  \hbox{\raise.1ex\hbox to 1\wd7{\hss.\hss}}\penalty 10000 \hskip-1\wd7\penalty
  10000\box7} \def\cprime{$'$}


 \end{document}